\documentclass[11pt,reqno]{amsart}
\usepackage{amssymb,amsmath,amsthm,amsfonts}
\usepackage[english]{babel}
\usepackage{mathrsfs,a4wide,dsfont}
\usepackage[utf8]{inputenc}
\usepackage{color}
\usepackage{comment}
\usepackage{bm}
\usepackage{subcaption}
\usepackage{graphicx}
\graphicspath{{fig/}}

\theoremstyle{plain}
\newtheorem{theorem}{Theorem}[section]

\newtheorem{remark}[theorem]{Remark}

\theoremstyle{definition}
\theoremstyle{remark}
\numberwithin{equation}{section}

\newcommand{\R}{{\mathbb R}}

\newcommand{\N}{{\mathbb N}}

\renewcommand{\k}{\kappa}

    \let\TeXchi\chi
\newbox\chibox
\setbox0 \hbox{\mathsurround0pt $\TeXchi$}
\setbox\chibox \hbox{\raise\dp0 \box 0 }
\def\chi{\copy\chibox}

\title[Variational analysis of inextensible elastic curves]{Variational analysis of inextensible elastic curves}

\author[G.\,Bevilacqua]
{Giulia Bevilacqua}
\address[G.\,Bevilacqua]{MOX-Dipartimento di Matematica, Politecnico di Milano, via Bonardi 9, I-20133 Milano, Italy.}
\email{giulia.bevilacqua@polimi.it - giulia.bevilacqua1993@gmail.com}

\author[L.\,Lussardi]
{Luca Lussardi}
\address[L.\,Lussardi]{Dipartimento di Scienze Matematiche ``G.L.\,Lagrange'', Politecnico di Torino, c.so Duca degli Abruzzi 24, I-10129 Torino, Italy.}
\email[]{luca.lussardi@polito.it}

\author[A.\,Marzocchi]
{Alfredo Marzocchi}
\address[A.\,Marzocchi]{Dipartimento di Matematica e Fisica ``N.\,Tartaglia'', Universit\`a Cattolica del Sacro Cuore, via Musei 41, I-20121 Brescia, Italy.}
\email[]{alfredo.marzocchi@unicatt.it}

\begin{document}

\baselineskip3.4ex

\vspace{0.5cm}
\begin{abstract}
{\small We minimize elastic energies on framed curves which penalize both curvature and torsion. We also discuss critical points using the infinite dimensional version of the Lagrange multipliers' method. Finally, some examples arising from the applications are discussed and also numerical experiments are presented. 
\vskip .3truecm
\noindent Keywords: Framed curves, inextensible curves, curvature, torsion, critical points.
\vskip.1truecm
\noindent 2010 Mathematics Subject Classification: 49J05, 46G05, 53A04, 74B20.}
\end{abstract}

\maketitle

\section{Introduction}
The study of elastic curves was initiated in 1691 by Jacob Bernoulli and it was continued by Euler who introduced, in his book of 1744, the methods of Calculus of Variations. In his masterpiece, Euler introduced a complete characterization of {\em elasticae curves}. Since then, the name {\em elasticae} refers to curves which are critical points for the energy functional
\[
\int_{\bm r} \k^2\,d\ell\,,
\]
where $\k$ is the curvature of the curve ${\bm r}$. Since the fundamental papers by Langer and Singer \cite{LS83,LS,LS_85}, where the equations of elasticae have been integrated, the study of the Euler functional has been very vast. We refer here, for instance, to a recent and very interesting research on the elastic networks (see \cite{ANP,GMP} and references therein). 

Elastic energies play an important role in physical applications: we just mention, for instance, the study of slender biological systems, like DNA, knotted or unknotted proteins \cite{CS,FNS}, or the construction of engineering structures, like cables or pipelines \cite{SH_nature}. 

In the literature, we can find also energy functionals that penalize both curvature and torsion. For instance, in 1930 Sadowsky \cite{S1,S2} (see \cite{HF1,HF2} for an English translation) studied the equilibria of a developable M\"obius strip by minimizing the bending energy. He argued that when the M\"obius strip shrinks to its centerline, the energy reduces to a functional which depends on the curvature and the torsion of the centerline itself. The original form of the Sadowsky energy functional is given by 
\[
\int_{\bm r} \frac{(\k^2+\tau^2)^2}{\k^2}\,d\ell\,,
\]
where $\tau$ is the torsion of ${\bm r}$. Later, also Langer and Singer \cite{LS_96} considered an energy functional which penalizes both the curvature and the torsion of the centerline of an elastic rod. Precisely, they considered the functional
\[
\lambda_1 \int_{\bm r}\,d\ell+\lambda_2\int_{\bm r}\tau\,d\ell+\lambda_3\int_{\bm r} \frac{\k^2}{2}\,d\ell\,.
\]

In this paper, we want to study a more general energy density function depending both on the curvature $\k$ and the torsion $\tau$ of the curve. In other words, we are dealing with the following type of elastic energy functional:
\begin{equation}\label{F}
\int_{\bm r} f(\k,\tau)\,d\ell.
\end{equation}
We are interested in the existence of minimizers of \eqref{F} among closed curves with fixed length, and in a characterization of its critical points. The corresponding Euler-Lagrange equations without constraints have already been obtained but, up to our knowledge, for curves of class at least $C^2$: for instance, in \cite{CCG, SH} the authors employ the Serret-Frenet frame to describe the geometry of the curves and to compute the first variation. Moreover, we obtain a system of first-order differential equations which is not in normal form but embedding the closure of the curve, a fact generally difficult to implement.

Differently from classical approaches \cite{LS_85}, where the independent variable of the energy functional is a parametrized curve and in order to introduce a model for an elastic curve as physical as possible \cite{dill_protein}, here we require that the elastic curve is $C^1$ but not necessarily $C^2$. By this, we cannot adopt the Serret-Frenet frame description, for which the regularity of the curve has to be at least $C^2$. We adopt here the approach of the {\em framed curves} with the constraints of $C^1$-closedness and a natural condition $\bm{t}' \cdot \bm{b}=0$. Framed curves were introduced by Schuricht et al.\,\cite{GMSM} to describe the physical behavior of elastic curves under additional topological constraints (see also \cite{Schuricht}), while the condition $\bm{t}' \cdot \bm{b}=0$ arises in a natural way from a Gamma-limit procedure exploited by Freddi et al.\, \cite{FHMP} to investigate the dimensional reduction of an elastic M\"obius strip. 

The paper is organized as follows. First of all, in Section \ref{sec:framed}, we introduce the mathematical setting of the framed curves showing how to reconstruct a space curve starting from its (weak) curvature and torsion and we introduce the elastic energy functional. Next, in Section \ref{sec:minimizer}, we prove the first main result, i.e.\,the existence of energy minimizers. Then, in Section \ref{sec:critical_points}, we find as a second result, the first-order necessary conditions for minimizers using essentially the infinite-dimensional version of the Lagrange multipliers' method. Finally, in Section \ref{sec:examples}, we consider some  examples arising from biological or engineering applications and we perform numerical examples to visualize the shape of critical points.

The variational analysis of functionals of type \eqref{F} is a fundamental preliminary in view to consider the more complicated physical situation where a soap film spans an elastic inextensible curve. This study will be managed in the spirit of the one carried out for the Kirchhoff-Plateau problem (see \cite{GFF,GLF,BLM1,BLM2,BLM3} and references therein) and it will be the content of a forthcoming paper.

\section{Framed curves and elastic energy}
\label{sec:framed}

We introduce framed curves following, up to some variants, the approach presented in \cite{GMSM}. We denote by $SO(3)$ the set of all $3\times 3$ rotation matrices: in other words, $({\bm u}|{\bm v}|{\bm w})\in SO(3)$ means that $\{{\bm u},{\bm v},{\bm w}\}$ is a positively oriented orthonormal basis of $\R^3$. Fix $L>0$ and $p>1$. On a triple $({\bm t}|{\bm n}|{\bm b})\in  W^{1,p}((0,L);SO(3))$ we put the following constraints: 
\begin{equation}\label{c1} 
{\bm t}'\cdot {\bm b}=0, \quad \text{a.e.\,on $(0,L)$},
\end{equation}
\begin{equation}\label{c2} 
\int_0^L{\bm t}\,ds=0,
\end{equation}
\begin{equation}\label{c3} 
{\bm t}(L)={\bm t}(0),
\end{equation}
where ${\bm t}(L)$ and ${\bm t}(0)$ are intended in the sense of traces. We let 
\[
W=\{({\bm t}|{\bm n}|{\bm b})\in  W^{1,p}((0,L);SO(3)) : \text{\eqref{c1}-\eqref{c2}-\eqref{c3} hold true}\}.
\]
Let $f\colon [0,L] \times \R \times \R \to\R$ be measurable. We define the energy functional $\mathcal E\colon W \to\R\cup\{+\infty\}$ as
\[
\mathcal E({\bm t}|{\bm n}|{\bm b})=\int_0^Lf(s,{\bm t}'\cdot {\bm n},{\bm n}'\cdot {\bm b})\,ds.
\] 

\subsection{Geometrical interpretation}

Fix $({\bm t}|{\bm n}|{\bm b})\in W$ and ${\bm x}_0 \in \R^3$. We consider the map ${\bm r}\colon [0,L]\to \R^3$ defined by 
\[
{\bm r}(s)={\bm x}_0+\int_0^s{\bm t}\,dr.
\]
In other words, ${\bm r}$ is the curve clamped at the point ${\bm x}_0$ and {\it generated} by the orthonormal frame $\{{\bm t},{\bm n},{\bm b}\}$. First of all ${\bm r}$ is parametrized by the arclength since $|{\bm r}'|=|{\bm t}|=1$. Condition \eqref{c2} says that ${\bm r}$ is a closed curve, that is ${\bm r}(0)={\bm r}(L)$. Moreover, condition \eqref{c3} says that the tangent vector to ${\bm r}$ is continuous, that is ${\bm r}'(0)={\bm r}'(L)$. We also point out that ${\bm r}$ belongs to $W^{2,p}((0,L);\R^3)$ but, in general, it does not belong to $W^{3,p}((0,L);\R^3)$ even if ${\bm n}\in W^{1,p}((0,L);\R^3)$. Hence, the regularity of admissible curves is in between $C^1$ and $C^2$. Finally, we let 
\[
\k={\bm t}'\cdot {\bm n}, \quad \tau={\bm n}'\cdot {\bm b}.
\]
Notice that $\k,\tau\in L^p(0,L)$. It is easy to see that condition \eqref{c1} implies that the following system holds 
\begin{equation}\label{ode}
\left\{\begin{array}{lc}
{\bm t}'=\k{\bm n},\\
{\bm n}'=-\k{\bm t}+\tau{\bm b},\\
{\bm b}'=-\tau{\bm n},
\end{array}\right.
\end{equation}
which looks like the usual {\it Serret-Frenet system} of ${\bm r}$. For that reason, $\k$ and $\tau$ can be regarded as the {\it (signed) weak curvature} of ${\bm r}$ and the {\it weak torsion} of ${\bm r}$ respectively. However, we remark that while in classical Differential Geometry the torsion of a space curve is not defined at a point where the curvature vanishes, for us the quantity ${\bm n}'\cdot {\bm b}$ is always defined.

\section{Minimizers of $\mathcal E$}
\label{sec:minimizer}

Our first main result is the existence of minimizers for $\mathcal E$. 

\begin{theorem}\label{main1}
Assume that: 
\begin{equation}\label{f0}
\text{$f(\cdot,a,b)\in L^1(0,L)$ for any $a,b\in \R$},
\end{equation}
\begin{equation}\label{f1}
\text{$f(s,\cdot)$ is continuous and convex for any $s\in [0,L]$}, 
\end{equation}
\begin{equation}\label{f2}
\text{$f(s,a,b)\geq c_1|a|^p+c_2|b|^p+c_3$ for any $a,b\in \R$},
\end{equation}
for some $c_1,c_2>0$, $c_3\in \R$. Then $\mathcal E$ has a minimizer on $W$.
\end{theorem}

\begin{proof}
We divide the proof into three steps.
\\
\\
{\sl Step 1}. We claim that $\inf_W\mathcal E<+\infty$. Consider
\[
{\bm t}^*(s)=-\sin \left(\frac{2\pi s}{L}\right)\,{\bm e}_1+\cos \left(\frac{2\pi s}{L}\right)\,{\bm e}_2,
\]
\[
{\bm n}^*(s)=-\cos \left(\frac{2\pi s}{L}\right)\,{\bm e}_1-\sin \left(\frac{2\pi s}{L}\right)\,{\bm e}_2,
\]
and 
\[
{\bm b}^*(s)={\bm e}_3
\]
where $\{{\bm e}_1,{\bm e}_2,{\bm e}_3\}$ is the canonical basis of $\R^3$. It is easy to see that $({\bm t}^*|{\bm n}^*|{\bm b}^*)\in W$. Moreover, 
\[
\k=({\bm t}^*)'\cdot {\bm n}^*=\frac{4\pi^2}{L^2}, \quad \tau=({\bm n}^*)'\cdot {\bm b}^*=0.
\]
Hence 
\[
\mathcal E({\bm t}^*|{\bm n}^*|{\bm b}^*)=\int_0^Lf\left(s,\frac{4\pi^2}{L^2},0\right)\,ds\stackrel{\eqref{f0}}{<}+\infty
\]
from which the claim follows.
\\
\\
{\sl Step 2}. We prove that $W$ is sequentially closed with respect to the weak convergence of $W^{1,p}((0,L);\R^{3\times 3})$. In order to prove this, let $({\bm t}_h|{\bm n}_h|{\bm b}_h)$ be a sequence in $W$ that converges weakly in $W^{1,p}((0,L);\R^{3\times 3})$ to $({\bm t}|{\bm n}|{\bm b})$. In particular, $\|{\bm t}_h\|_{1,p},\|{\bm n}_h\|_{1,p},\|{\bm b}_h\|_{1,p}$ are bounded. Taking into account the Sobolev compact embedding $W^{1,p}((0,L);\R^3) \hookrightarrow C^0([0,L];\R^3)$ we can say that, up to a subsequence not relabeled, $({\bm t}_h|{\bm n}_h|{\bm b}_h) \to ({\bm t}|{\bm n}|{\bm b})$ uniformly on $[0,L]$. Now, since for any $s\in [0,L]$ we have $({\bm t}_h(s)|{\bm n}_h(s)|{\bm b}_h(s))\in SO(3)$ and $SO(3)$ is closed in $\R^{3\times 3}$ we deduce that $({\bm t}(s)|{\bm n}(s)|{\bm b}(s))\in SO(3)$ for any $s\in [0,L]$ since the uniform convergence implies the pointwise convergence. For the same reason condition \eqref{c2} is preserved in the limit as well as the constraint \eqref{c3}. It remains to prove that \eqref{c1} passes to the limit. Since ${\bm t}_h'\to {\bm t}'$ weakly in $L^p(0,L)$ and ${\bm b}_h \to {\bm b}$ uniformly on $[0,L]$ we can say that ${\bm t}_h'\cdot {\bm b}_h \to {\bm t}'\cdot {\bm b}$ weakly in $L^p(0,L)$. Hence 
\[
\|{\bm t}'\cdot {\bm b}\|_p \le \liminf_{h\to+\infty}\|{\bm t}_h'\cdot {\bm b}_h\|_p=0
\]
from which we obtain ${\bm t}'\cdot {\bm b}=0$.
\\
\\
{\sl Step 3}. The proof now uses the Direct Method of the Calculus of Variations. Let $({\bm t}_h|{\bm n}_h|{\bm b}_h)$ be a minimizing sequence for $\mathcal E$ on $W$, that is $({\bm t}_h|{\bm n}_h|{\bm b}_h)\in W$ for any $h\in \N$ and 
\[
\lim_{h\to+\infty}\mathcal E({\bm t}_h|{\bm n}_h|{\bm b}_h)=\inf_W\mathcal E.
\]
Since $\inf_W\mathcal E<+\infty$, and by \eqref{f2} we can say that $\|({\bm t}_h|{\bm n}_h|{\bm b}_h)\|_{1,p}$ is bounded. Then, up to a subsequence (not relabeled), we get $({\bm t}_h|{\bm n}_h|{\bm b}_h) \rightharpoonup  ({\bm t}|{\bm n}|{\bm b}) \in W$ because of the weak closure of $W$. As before, notice that ${\bm t}_h'\cdot {\bm n}_h \to {\bm t}'\cdot {\bm n}$ and ${\bm n}_h'\cdot {\bm b}_h \to {\bm n}'\cdot {\bm b}$ both weakly in $L^p(0,L)$. Since condition \eqref{f1} guarantees the lower semicontinuity of $\mathcal E$ with respect to the weak topology of $L^p(0,L)$ we conclude that 
\[
\mathcal E({\bm t}|{\bm n}|{\bm b})\le \liminf_{h\to+\infty}\mathcal E({\bm t}_h|{\bm n}_h|{\bm b}_h)=\inf_W\mathcal E
\]
which ends the proof. 
\end{proof}

\section{Critical points of $\mathcal E$}\label{sec:critical_points}

In this section we want to find the first-order necessary conditions for minimizers of $\mathcal E$. In order to do this, we recall the infinite-dimensional version of the Lagrange multipliers' method. In what follows, $Y'$ denotes the topological dual of $Y$.

\begin{theorem}\label{Lag}
Let $X,Y$ be two real Banach spaces, $\mathcal F\in C^1(X)$ and $\mathcal G \in C^1(X;Y)$. Assume that $\mathcal G'(x)\ne 0$ whenever $\mathcal G(x)=0$. Let $x_0 \in X$ be such that 
\[
\mathcal F(x_0)=\min\{\mathcal F(x) : \mathcal G(x_0)=0\}. 
\]
Then there exists a Lagrange multiplier $\lambda \in Y'$ such that
\[
\mathcal F'(x_0)=\lambda(\mathcal G'(x_0)).
\]
\end{theorem}

We are ready to state and prove our second main result.

\begin{theorem}\label{main2}
Assume that $f$ is of class $C^1$ and satisfies 
\begin{equation}\label{f3}
f(s,a,b)\le c(1+|a|^p+|b|^p)
\end{equation}
and
\begin{equation}\label{f4}
|f_a(s,a,b)|\le c(1+|a|^{p-1}+|b|^{p-1}),\quad |f_b(s,a,b)|\le c(1+|a|^{p-1}+|b|^{p-1})
\end{equation}
for all $s\in [0,L]$ and any $a,b\in \R$ and for some $c\ge 0$ (here $f_a=\frac{\partial f}{\partial a}$ and $f_b=\frac{\partial f}{\partial b}$). Let $({\bm t}|{\bm n}|{\bm b})\in W$ be a minimizer of $\mathcal E$ and let $\k={\bm t}'\cdot {\bm n},\tau={\bm n}'\cdot {\bm b}$. Then, $f_a,f_b \in W^{1,1}(0,L)$ and there exist $\mu \in L^{p'}(0,L)$ with $\mu'\in L^p(0,L)$ and ${\bm \lambda}\in \R^3$ such that the following conditions hold a.e.\,on $(0,L)$:
\begin{equation}\label{elw}
\left\{\begin{array}{lll}
f_b(s,\k,\tau)'=\mu\k\\
-f_a(s,\k,\tau)'=\mu\tau+{\bm\lambda} \cdot {\bm n}\\
\k f_b(s,\k,\tau)-\tau f_a(s,\k,\tau)=-\mu'+{\bm\lambda} \cdot {\bm b}.
\end{array}\right.
\end{equation}
\end{theorem}

\begin{proof}
Let $X=W^{1,p}((0,L);\R^{3\times 3})$. The free variable in $X$ will be denoted again by $({\bm t}|{\bm n}|{\bm b})$. We define the functional $\mathcal F \colon X \to \R$ by 
\[
\mathcal{F}({\bm t}|{\bm n}|{\bm b}) =\int_0^L f(s,{\bm t}'\cdot {\bm n},{\bm n}'\cdot {\bm b})\,ds.
\]
Fix $({\bm{\bm \eta}}_1|{\bm {\bm \eta}}_2|{\bm {\bm \eta}}_3) \in X$. Thanks to \eqref{f4} we can differentiate under the integral sign obtaining 
\[
\begin{aligned}
&\frac{d}{d\sigma}\mathcal F({\bm t}+\sigma{\bm \eta}_1|{\bm n}+\sigma{\bm \eta}_2|{\bm b}+\sigma{\bm \eta}_3)_{|_{\sigma=0}}\\
&=\int_0^Lf_a(s,{\bm t}'\cdot {\bm n},{\bm n}'\cdot {\bm b})({\bm n}\cdot {\bm \eta}'_1+{\bm t}'\cdot {\bm \eta}_2)\,ds+\int_0^L f_b(s,{\bm t}'\cdot {\bm n},{\bm n}'\cdot {\bm b})({\bm b}\cdot {\bm \eta}'_2+{\bm n}'\cdot {\bm \eta}_3)\,ds\\
&\qquad =:L({\bm \eta}_1,{\bm \eta}_2,{\bm \eta}_3),
\end{aligned}
\]
where $L$ is a linear operator. It is easy to see that 
\[
|L({\bm \eta}_1|{\bm \eta}_2|{\bm \eta}_3)|\le m\|({\bm t}|{\bm n}|{\bm b})\|_X\|({\bm \eta}_1|{\bm \eta}_2|{\bm \eta}_3)\|_X
\]
for a suitable constant $m>0$ by \eqref{f4}, the H\"older inequality and the continuous embeddings 
\[
W^{1,p}((0,L);\R^3) \hookrightarrow C^0([0,L];\R^3), \quad L^p((0,L);\R^3) \hookrightarrow L^1((0,L);\R^3).
\] 
Therefore, $L$ is also continuous and 
\[
|L({\bm \eta}_1|{\bm \eta}_2|{\bm \eta}_3)|\le m\|({\bm t}|{\bm n}|{\bm b})\|_X
\]
whenever $\|({\bm \eta}_1|{\bm \eta}_2|{\bm \eta}_3)\|_X=1$. Applying standard results of Nonlinear Analysis (see for instance \cite[Sec.\,1.3]{BS}), we can then conclude that $\mathcal F \in C^1(X)$ and 
\[
\begin{aligned}
&\mathcal F'({\bm t}|{\bm n}|{\bm b})({\bm \eta}_1|{\bm \eta}_2|{\bm \eta}_3)\\
&=\int_0^Lf_a(s,{\bm t}'\cdot {\bm n},{\bm n}'\cdot {\bm b})({\bm n}\cdot {\bm \eta}'_1+{\bm t}'\cdot {\bm \eta}_2)\,ds+\int_0^L f_b(s,{\bm t}'\cdot {\bm n},{\bm n}'\cdot {\bm b})({\bm b}\cdot {\bm \eta}'_2+{\bm n}'\cdot {\bm \eta}_3)\,ds
\end{aligned}
\]
for any $({\bm \eta}_1|{\bm \eta}_2|{\bm \eta}_3) \in X$. Next, we consider the constraints. We let 
\[
Y=L^p(0,L) \times L^p(0,L) \times L^p(0,L) \times L^p((0,L);\R^3)\times L^p(0,L) \times \R^3 \times \R^3
\]
equipped with the product topology in order to get a Banach space. We define $\mathcal G \colon  X\to Y$ as 
\[
\mathcal G({\bm t}|{\bm n}|{\bm b})=\left({\bm t} \cdot {\bm t}-1,{\bm n} \cdot {\bm n}-1,{\bm t}\cdot {\bm n},{\bm b}-{\bm t}\times {\bm n},{\bm t}' \cdot {\bm b},\int_0^L {\bm t}\,ds,{\bm t}(L)-{\bm t}(0)\right).
\]
Using the same argument as before, we can easily see that $\mathcal G \in C^1(X,Y)$ and  
\[
\begin{aligned}
\mathcal G'&({\bm t}|{\bm n}|{\bm b})({\bm \eta}_1|{\bm \eta}_2|{\bm \eta}_3)\\
&=\bigg(2{\bm t}\cdot {\bm \eta}_1,2{\bm b}\cdot {\bm \eta}_2,{\bm t}\cdot {\bm \eta}_2+{\bm n}\cdot {\bm \eta}_1,{\bm \eta}_3+{\bm n}\times {\bm \eta}_1-{\bm t}\times {\bm \eta}_2,{\bm b} \cdot {\bm \eta}'_1+{\bm t}'\cdot {\bm \eta}_3,\\
&\qquad \quad \int_0^L {\bm \eta}_1\,ds,{\bm \eta}_1(L)-{\bm \eta}_1(0)\bigg).
\end{aligned}
\]
Moreover, $\mathcal G'({\bm t}|{\bm n}|{\bm b})\ne 0$ for any $({\bm t}|{\bm n}|{\bm b}) \in X$ such that $\mathcal G({\bm t}|{\bm n}|{\bm b})=0$. Then, by construction a minimizer of $\mathcal E$ is a constrained minimizer of $\mathcal F$ on $\{\mathcal G=0\}$. From now on, $({\bm t}|{\bm n}|{\bm b})$ denotes such a minimizer. For simplicity of notation we also let $f_a=f_a(s,\k,\tau)$ and $f_b=f_b(s,\k,\tau)$. Applying Theorem \ref{Lag}, we can say that there exist Lagrange multipliers $\lambda_1,\lambda_2,\lambda_3 \in L^{p'}(0,L)$, ${\bm \lambda}_4\in L^{p'}((0,L);\R^3)$, $\mu\in L^{p'}(0,L)$, ${\bm \lambda}\in \R^3$ such that
\begin{equation}\label{fond}
\begin{aligned}
&\int_0^Lf_a{\bm n}\cdot {\bm \eta}'_1\,ds+\int_0^L f_a{\bm t}'\cdot {\bm \eta}_2\,ds+\int_0^L f_b{\bm b}\cdot {\bm \eta}'_2\,ds+\int_0^L f_b {\bm n}' \cdot {\bm \eta}_3\,ds\\
&=\int_0^L(2\lambda_1{\bm t}+\lambda_3{\bm n}+{\bm \lambda}_4\times {\bm n}+{\bm \lambda})\cdot {\bm \eta}_1\,ds+\int_0^L\mu {\bm b}\cdot {\bm \eta}'_1\,ds\\
&\quad +\int_0^L(2\lambda_2{\bm n}+\lambda_3{\bm t}-{\bm \lambda}_4\times {\bm t})\cdot {\bm \eta}_2\,ds+\int_0^L({\bm \lambda}_4+\mu {\bm t}')\cdot {\bm \eta}_3\,ds
\end{aligned}
\end{equation}
for any $({\bm \eta}_1|{\bm \eta}_2|{\bm \eta}_3) \in W^{1,p}_0((0,L);\R^{3\times 3})$. Using ${\bm \eta}_1={\bm \eta}_2=0$ and the arbitrariness of ${\bm \eta}_3$ we easily obtain $$
{\bm \lambda}_4=-\mu {\bm t}'+ f_b{\bm n}'\,.
$$
Now, using ${\bm \eta}_1={\bm \eta}_3=0$ and ${\bm \eta}_2=\varphi{\bm b}$ with $\varphi\in C_c^1(0,L)$ we deduce that 
\[
-f_b'={\bm \lambda}_4 \cdot {\bm n}=-\mu\,{\bm t}'\cdot {\bm n}
\]
which is \eqref{elw}$_1$. Next, taking ${\bm \eta}_1={\bm \eta}_3=0$ and ${\bm \eta}_2=\varphi{\bm t}$ we easily get 
$$
\lambda_3=0.
$$
Finally, considering ${\bm \eta}_2={\bm \eta}_3=0$ and ${\bm \eta}_1=\varphi{\bm n}$ or ${\bm \eta}_1=\varphi{\bm b}$ we arrive at 
\[
-f_a'={\bm \lambda}\cdot {\bm n}+\mu\,{\bm n'}\cdot {\bm b},
\]
which is  \eqref{elw}$_2$ and
\[
({\bm t}'\cdot {\bm n})f_b-({\bm n}'\cdot {\bm b})f_a=-\mu'+{\bm\lambda} \cdot {\bm b}.
\]
which is exactly \eqref{elw}$_3$, and the proof is complete.
\end{proof}

\begin{remark}{\rm
We point out that from \eqref{fond} we can deduce other conditions that permit us to find $\lambda_1$ and $\lambda_2$, useless for our purposes. For the sake of completeness we give the complete system of conditions obtained from \eqref{fond}:
\[
\left\{\begin{array}{lr}
({\bm t}'\cdot {\bm n})f_a=2\lambda_1-{\bm \lambda}_4\cdot {\bm b}+{\bm \lambda}\cdot {\bm t}, & \text{choosing ${\bm \eta}_1=\varphi{\bm t},{\bm \eta}_2={\bm \eta}_3=0$},\\
-f_a'=\lambda_3+{\bm \lambda}\cdot {\bm n}+\mu{\bm n'}\cdot {\bm b}, & \text{choosing ${\bm \eta}_1=\varphi{\bm n},{\bm \eta}_2={\bm \eta}_3=0$},\\
\mu'={\bm \lambda}_4\cdot {\bm t}+{\bm \lambda}\cdot {\bm b}+({\bm n}'\cdot {\bm b})f_a, & \text{choosing ${\bm \eta}_1=\varphi{\bm b},{\bm \eta}_2={\bm \eta}_3=0$},\\
\lambda_3=0, & \text{choosing ${\bm \eta}_2=\varphi{\bm t},{\bm \eta}_1={\bm \eta}_3=0$},\\
({\bm t}'\cdot{\bm n})f_a+({\bm n}'\cdot{\bm b})f_b=2\lambda_2-{\bm \lambda}_4 \cdot {\bm t}, & \text{choosing ${\bm \eta}_2=\varphi{\bm n},{\bm \eta}_1={\bm \eta}_3=0$},\\
-f_b'={\bm\lambda}_4\cdot {\bm n}, & \text{choosing ${\bm \eta}_2=\varphi{\bm b},{\bm \eta}_1={\bm \eta}_3=0$},\\
 f_b{\bm n}'={\bm \lambda}_4+\mu {\bm t}', & \text{choosing ${\bm \eta}_3$ arbitrarily and ${\bm \eta}_1={\bm \eta}_2=0$}.
\end{array}\right.
\]
}
\end{remark}

We conclude this general section with the complete elimination of the Lagrange multipliers in the system \eqref{elw}, but we have to assume a priori regularity. 

\begin{theorem}\label{main3}
Assume that $f$ is of class $C^3$ and let $({\bm t}|{\bm n}|{\bm b}) \in W$ be a solution of the system \eqref{elw} with ${\bm t},{\bm n},{\bm b}$ of class $C^4$. Then at any point where $\k\ne 0$ we have 
\begin{equation}\label{reg1}
\left\{\begin{array}{ll}
\displaystyle 2(\tau f_a)'-\tau' f_a-\left(\frac{f_b'}{\k}\right)''+\frac{\tau^2}{\k}f'_b-(\k f_b)'=0\\
\\
\displaystyle -\k f_a'-\tau f_b'=\left(\frac{f_a''}{\k}-\frac{\tau^2}{\k}f_a+\frac{2\tau}{\k}\left(\frac{f_b'}{\k}\right)'+\tau'\frac{f_b'}{\k^2}+\tau f_b\right)'
\end{array}\right.
\end{equation}
where $f_a=f_a(s,\k,\tau)$ and  $f_b=f_b(s,\k,\tau)$.
\end{theorem}

\begin{proof}
Let us take a point where $\k\ne 0$. From \eqref{elw}$_1$ we get 
\[
\mu=\frac{f_b'}{\k}.
\]
Now, differentiating \eqref{elw}$_3$, using the fact that ${\bm b}'=-\tau {\bm n}$ and inserting \eqref{elw}$_2$ we easily get
\[
\left(\frac{f_b'}{\k}\right)''=-\k'f_b-\k f_b'+\tau'f_a+\frac{\tau^2}{\k}f'_b+2\tau f_a'
\]
which gives \eqref{reg1}$_1$. Next, differentiating \eqref{elw}$_2$, using the fact that ${\bm n}'=-\k{\bm t}+\tau {\bm b}$ and inserting \eqref{elw}$_3$ we obtain 
\[
{\lambda}\cdot {\bm t}=\frac{f_a''}{\k}-\frac{\tau^2}{\k}f_a+\frac{2\tau}{\k}\left(\frac{f_b'}{\k}\right)'+\tau'\frac{f_b'}{\k^2}+\tau f_b
\]
and then, since ${\bm t}'=\k{\bm n}$, 
\[
-\k f_a'-\tau f_b'=\left(\frac{f_a''}{\k}-\frac{\tau^2}{\k}f_a+\frac{2\tau}{\k}\left(\frac{f_b'}{\k}\right)'+\tau'\frac{f_b'}{\k^2}+\tau f_b\right)'
\]
which is \eqref{reg1}$_2$.
\end{proof}

\section{Some examples}
\label{sec:examples}

In this section we discuss some explicit examples arising from the applications.

\subsection{The Euler elastica}

As recalled in the introduction, the study of the Euler elastica functional 
\[
\int_0^L \k^2\,ds
\]
has been very vast. We refer here to the papers by Langer and Singer \cite{LS83,LS,LS_85}. First of all, they proved that the unique global (and local) minimizer among all $W^{2,2}$ curves $C^1$-periodics is the circumference with length $L$. Concerning critical points, there is essentially another planar and $C^1$-periodic critical point, which is known in literature as {\it  lemniscate} (an eight-figure). A great variety of space critical points in the same class of admissible curves can be obtained and all of them lie on an embedded torus of revolution. Finally, many other critical points can be found if we do not assume the closedness of the curve ({\it free elasticae}). As proved by Langer and Singer \cite{LS}, the general equations of elasticae are given by  
\[
\left\{\begin{array}{ll}
2\k''-2\k\tau^2+\k^3-c_1\k=0\\
\k^2\tau=c_2
\end{array}\right.
\]
where $c_1,c_2$ are constants (free elasticae are obtained for $c_1=0$). 

Following our notation, we get the Euler functional choosing $f(s,a,b)=a^2$. In this case condition \eqref{f2} is not satisfied so we are not able to apply our Theorem \ref{main1} in order to get minimizers, at least for space curves. Nevertheless, conditions \eqref{f3} and \eqref{f4} are satisfied, so that, we are able to apply Theorem \ref{main2}, since we know that at least a minimizer exists. 

\begin{theorem}
Let $\mathcal E \colon W \to \R \cup \{+\infty\}$ be given by 
\[
\mathcal E({\bm t}|{\bm n}|{\bm b})=\int_0^L \k^2\,ds.
\]
Let $({\bm t}|{\bm n}|{\bm b}) \in W$ be a minimizer of $\mathcal E$. Let $S=\{s\in [0,L] : \k(s)=0\}$. Then $S$ is negligible. Moreover, $\k\in W^{2,2}(0,L)$, $\tau \in W^{1,2}(0,L)$ and there exist $c_1,c_2 \in \R$ such that 
\begin{equation}\label{elastica}
\left\{\begin{array}{ll}
2\k''-2\k\tau^2+\k^3-c_1\k=0\\
\k^2\tau=c_2
\end{array}\right.
\end{equation}
everywhere on $[0,L]$.
\end{theorem}

\begin{proof}
Since $f(s,a,b)=a^2$ we get $f_a=2a$ and $f_b=0$. We are in position to apply Theorem \ref{main2}. The system \eqref{elw} reads as 
\begin{equation}\label{el-1}
\left\{\begin{array}{lll}
\mu\k=0\\
{\bm\lambda} \cdot {\bm n}=-2\k'-\mu\tau\\
{\bm\lambda} \cdot {\bm b}=-2\k\tau+\mu'
\end{array}\right.
\end{equation}
for some $\mu \in W^{1,2}(0,L)$ and ${\bm \lambda} \in \R^3$. In particular, \eqref{el-1}$_2$ gives $\k\in W^{1,2}(0,L)$. Now we divide the proof into some steps.
\\
\\
{\it Step 1}. Since \eqref{el-1}$_1$ we get $\mu_{|_{[0,L]\setminus S}}=0$. Hence, from \eqref{el-1}$_3$ we obtain $\tau_{|_{[0,L]\setminus S}} \in W^{1,2}([0,L]\setminus S)$.
\\
\\
{\it Step 3}. We prove now that $\k\in W^{2,2}(0,L)$. Since $\k$ is continuous, $S$ is relatively closed in $[0,L]$. Hence, we can write it as 
\[
S=\bigcup_{h=0}^{+\infty} S_h
\] 
where $S_h$ are pairwise disjoint and $S_h$ is either a singleton or a closed interval with non-empty interior. On each $S_h=[a_h,b_h]$ with $a_h<b_h$ we change $\tau$ as follows: 
\[
\bar \tau (s)=\tau(a_h^-)+\frac{\tau(b_h^+)-\tau(a_h^-)}{b_h-a_h}(s-a_h), \quad \forall s \in [a_h,b_h],
\] 
where $\tau(a_h^-),\tau(b_h^+)$ are the left and right traces of $\tau$ respectively at $a_h$ and $b_h$, with the convention $\tau(0^-)=\tau(L^+)=0$. By construction, we obtain $\bar \tau \in W^{1,2}(0,L)$ and $\bar \tau=\tau$ on $[0,L]\setminus S$. Let $(\bar {\bm t}|\bar {\bm n}|\bar {\bm b}) \in W^{1,2}((0,L);\R^3)$ be the unique solution of the Cauchy problem
\[
\left\{\begin{array}{lc}
\bar {\bm t}'=\k\bar{\bm n}\\
\bar{\bm n}'=-\k\bar{\bm t}+\bar\tau\bar{\bm b}\\
\bar{\bm b}'=-\bar\tau\bar{\bm n}\\
\bar {\bm t}(0)={\bm t}(0)\\
\bar {\bm n}(0)={\bm n}(0)\\
\bar {\bm b}(0)={\bm b}(0).
\end{array}\right.
\]
It is easy to see that $(\bar {\bm t}|\bar {\bm n}|\bar {\bm b}) \in W$ (notice that actually $\bar{\bm t}={\bm t}$ everywhere). Moreover, $\mathcal E(\bar{\bm t}|\bar{\bm n}|\bar{\bm b})=\mathcal E({\bm t}|{\bm n}|{\bm b})$, hence $(\bar {\bm t}|\bar {\bm n}|\bar {\bm b})$ is still a minimizer of $\mathcal E$. Applying Theorem \ref{main2} again, we deduce that 
\begin{equation}\label{el-2}
\left\{\begin{array}{lll}
\bar\mu\k=0\\
\bar{\bm\lambda} \cdot \bar{\bm n}=-2\k'-\bar\mu\bar\tau\\
\bar{\bm\lambda} \cdot \bar{\bm b}=-2\k\bar\tau+\bar\mu'
\end{array}\right.
\end{equation}
for some $\bar \mu \in W^{1,2}(0,L)$ and $\bar{\bm \lambda} \in \R^3$. As a consequence of \eqref{el-2}$_2$ we obtain $\k\in W^{2,2}(0,L)$.
\\
\\
{\it Step 4.} We claim that for any relatively open interval $I \subseteq ([0,L]\setminus S)$ there exists $c_I \in \R$ such that 
\begin{equation}\label{elastica-I}
2\k''-2\k\tau^2+\k^3-c_I\k=0, \quad \text{on $I$.}
\end{equation}
First, on $I$ the system \eqref{el-1} reduces to 
\begin{equation}\label{euler0}
\left\{\begin{array}{ll}
{\bm\lambda} \cdot {\bm n}=-2\k'\\
{\bm\lambda} \cdot {\bm b}=-2\k\tau.
\end{array}\right.
\end{equation}
Differentiating \eqref{euler0}$_1$ we get ${\bm \lambda} \cdot {\bm n}'=-2\k''$. Since ${\bm n}'=-\k{\bm t}+\tau {\bm b}$ we obtain
\[
2\k''=-\k{\bm \lambda} \cdot {\bm t}+\tau {\bm \lambda} \cdot {\bm b} \in W^{1,2}(I),
\] 
from which $\k\in W^{3,2}(I)$. In particular, $\k\in C^2(\overline I)$. Combining ${\bm n}'=-\k{\bm t}+\tau {\bm b}$ with \eqref{euler0}$_2$ we deduce
\begin{equation}\label{eg}
\frac{2\k''}{\k}-2\tau^2={\bm \lambda}\cdot{\bm t}, \quad \text{on $I$.}
\end{equation}
As a consequence we obtain 
\[
\left(\frac{2\k''}{\k}-2\tau^2+\k^2\right)'={\bm \lambda}\cdot{\bm t}'+2\k\k'=\k{\bm \lambda}\cdot{\bm n}+2\k\k'=0
\]
where the last equality follows from \eqref{euler0}$_1$. Then, since $I$ is an interval, there exists $c_I \in \R$ such that 
\[
\frac{2\k''}{\k}-2\tau^2+\k^2=c_I, \quad \text{on $I$},
\]
which proves the claim.
\\
\\
{\it Step 5.} We prove that $S$ is negligible. First of all, it cannot be $\k=0$ everywhere, because of the constraint \eqref{c2}. In order to see that $S$ is negligible it is sufficient to show that in the decomposition of $\{S_h\}_{h\in\N}$ there is no $S_h$ with non-empty interior. Assume by contradiction that there exists $S_h=[a_h,b_h]$ with $a_h<b_h$. Then, either $a_h \in (0,L)$ or $b_h\in (0,L)$. Without loss of generality we can assume $b_h\in (0,L)$ (the argument for $a_h$ is the same). Then $\k \ne 0$ on $(b_h,b_h+\delta)$ for some $\delta>0$, so that $\k\in C^2([b_h,b_h+\delta])$ and using \eqref{elastica-I} we can say that 
\[
2\k''-2\k\tau^2+\k^3-c_h\k=0, \quad \text{on $(b_h,b_h+\delta)$}
\]
for some $c_h\in \R$. As a consequence we deduce that $\k$ is a solution of the Cauchy problem 
\[
\left\{\begin{array}{ll}
2\k''-2\k\tau^2+\k^3-c_h\k=0, \quad \text{on $(b_h,b_h+\delta)$}\\
\k(b_h)=0\\
\k'(b_h)=0
\end{array}
\right.
\]
which means that $\k=0$ on $(b_h,b_h+\delta)$, that is a contradiction.
\\
\\
{\it Step 6.} We can now conclude the proof. Since $S$ is negligible, we immediately deduce that $\tau \in W^{1,2}(0,L)$. It remains to show \eqref{elastica}. Observe that the function 
\[
\frac{2\k''}{\k}-2\tau^2+\k^2 
\]
is piecewise constant and it coincides with ${\bm \lambda}\cdot{\bm t}$ a.e.\,$t\in [0,L]$, As a consequence, 
\[
\frac{2\k''}{\k}-2\tau^2+\k^2 
\]
must be constant, which ends the proof.
\end{proof}

\subsubsection{Numerical results}
In this paragraph we show some numerical results obtained using the software Mathematica (Wolfram Inc., version 12) concerning solutions of the system \eqref{elastica}. We do separate analysis for planar curves and for space curves.
\\
\\
{\it Planar curves}. In this case $c_2=0$. First of all, we pass to the general Cauchy problem for $\k$, namely 
\begin{equation}\label{eq:system_num2d}
\left\{\begin{aligned}
&2 \k^3 \k''+\k^6 -c_1 \k^4=0,\\
&\k(0) = \k_0,\\
& \k'(0) = \k_1.
\end{aligned}
\right.
\end{equation}
The idea is to integrate numerically the system \eqref{eq:system_num2d} and then try to reconstruct the shape of the curve. Without loss of generality we can look for the curve ${\bm r}\colon [0,L] \to \R^2$
\[
{\bm r}(s)=\int_0^s{\bm t}\,dr
\] 
where the tangent vector ${\bm t}=(t^{(1)},t^{(2)})$ solves the system 
\begin{equation}\label{eq:ode_num2d}
\left\{\begin{array}{lc}
(t^{(1)})'=-\k t^{(2)}\\
(t^{(2)})'=\k t^{(1)}\\
t^{(1)}(0)=1\\
t^{(2)}(0)=0.
\end{array}\right.
\end{equation}
In other words, we require the curve to be clamped at the origin which ``starts'' with the canonical orthonormal frame. However, we point out that in general it is not necessarily true that ${\bm r}$ is admissible: for instance we do not have implemented any closedness of ${\bm r}$. From Differential Geometry, it is known that the necessary and sufficient conditions for a planar curve to be closed are the equalities
\[
\int_0^L\cos\left(\int_0^t \k(s)\,ds\right)\,dt=\int_0^L\sin\left(\int_0^t \k(s)\,ds\right)\,dt=0.
\]
We also anticipate the fact that for space curves there are no similar conditions on $\k,\tau$ in order to guarantee that the curve is closed. For details we refer to \cite{N}. Actually, for numerical reasons we decide to introduce a stop condition in the numeric integration of \eqref{eq:system_num2d}--\eqref{eq:ode_num2d}: we vary randomly the constants $c_1, \k_0,\k_1$ until the inequality 
\begin{equation}\label{eq:stopcond_piano}
	d = |\bm{r}(L)| +  |\bm{t}(L) - (1,0)| < 10^{-6}
\end{equation}
is satisfied. Condition \eqref{eq:stopcond_piano} formally implies that the obtained curve ${\bm r}$ is {\it almost closed} as well as its tangent vector.
\begin{figure}[h!]
	\begin{subfigure}{.45\linewidth}
		\centering
		\includegraphics[width=0.9\textwidth]{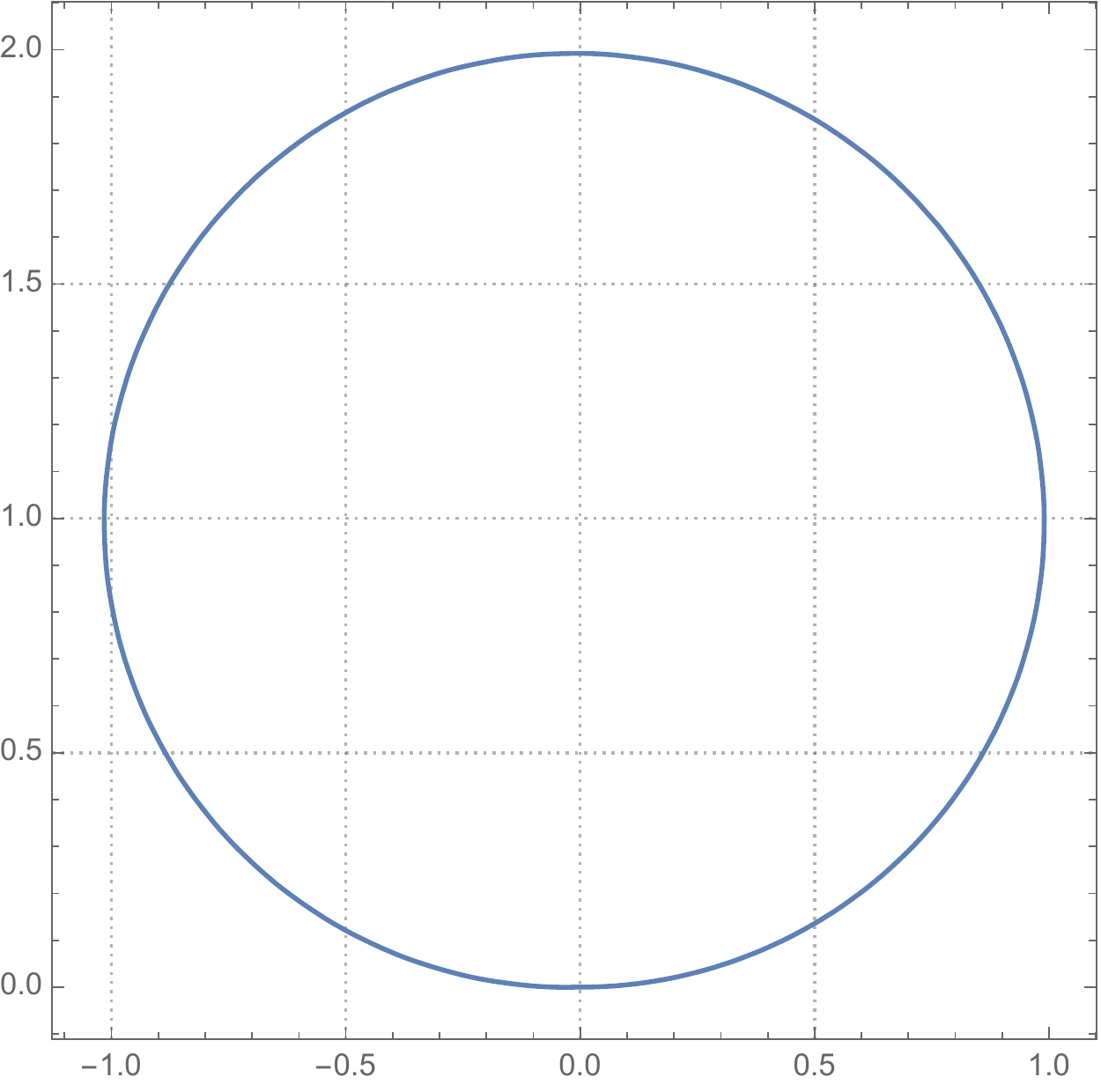}
		\caption{Circumference}
		\label{fig:fig3}
	\end{subfigure}
	\begin{subfigure}{.45\linewidth}
		\centering
		\includegraphics[width=\textwidth]{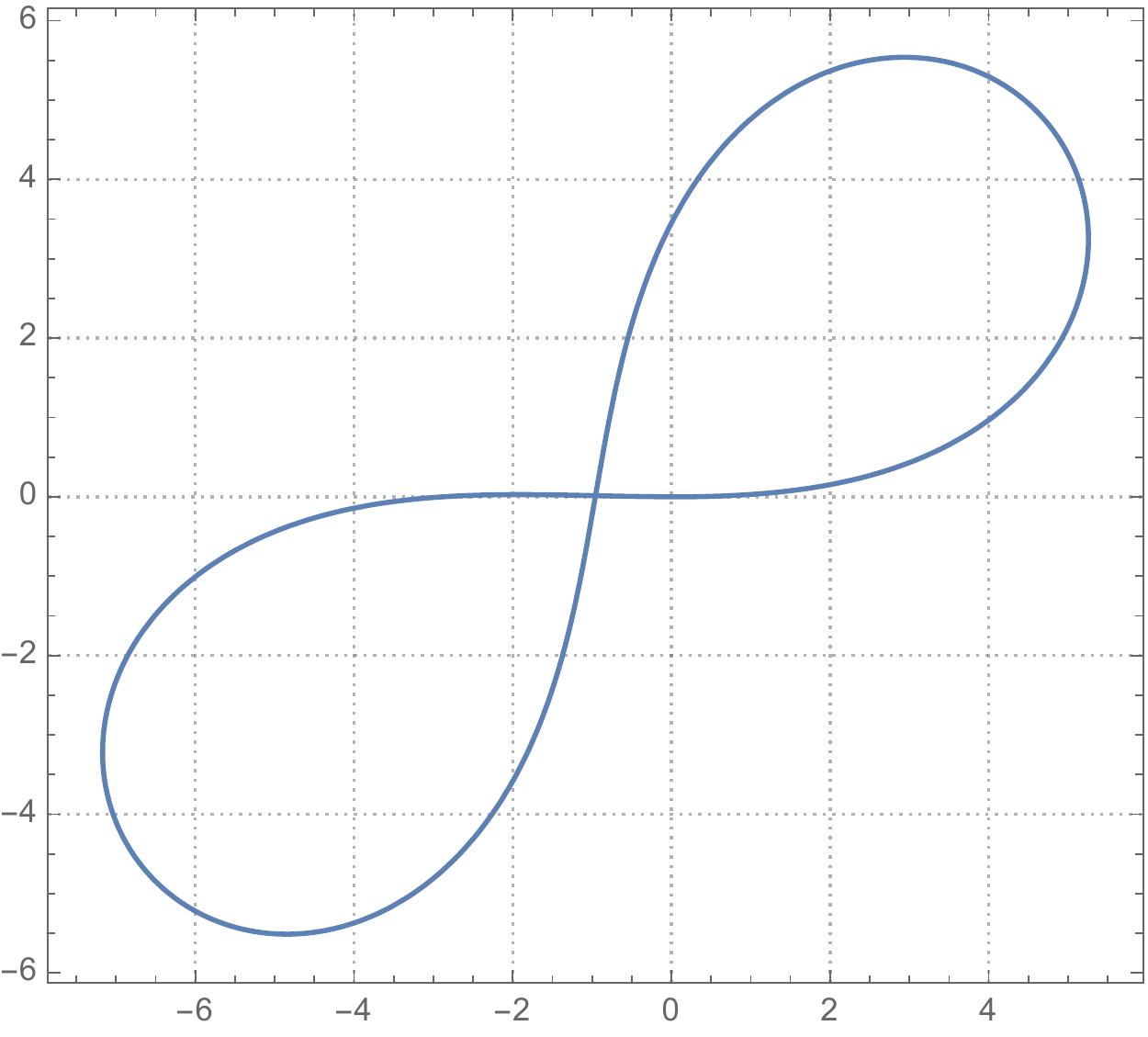}
		\caption{Lemniscate}
		\label{fig:fig7}
	\end{subfigure}
	\caption{Closed elasticae: the circumference (A) (the only one stable), and the lemniscate (B) (unstable, see \cite[Ex.\,3]{LS_85}).}
	\label{fig:k_closed}
\end{figure}

We obtain the circumference in Fig.\,\ref{fig:fig3} and the lemniscate in Fig.\,\ref{fig:fig7} with the values reported in Table \ref{tab:val_k_piano_closed}. 
\begin{table}[t!]
	\centering 
	\color{black}
	\begin{tabular}{|c|c|c|c|c|}
		\hline
		& $c_1$ &$\k_0$ &$\k_1$ &$L $ \\ \hline
	Circumference (Fig.\,\ref{fig:fig3})
		&$1.00824$
		&$1.01227$
		&$0.0003$
		&$2 \pi$
		\\ \hline
		Lemniscate (Fig.\,\ref{fig:fig7})
		&$0.07911031$
		&$0.0442$
		&$0.046801$
		&$12 \pi$
		\\ \hline
	\end{tabular}
	\caption{Numerical values for the circumference and the lemniscate.}\label{tab:val_k_piano_closed}
\end{table}

Without imposing condition \eqref{eq:stopcond_piano} we can find numerically a lot of free elasticae choosing in an arbitrary way the constants $c_1$, $\k_0$ and $\k_1$. In Table \ref{tab:val_k_open} we report a table for the chosen constants to plot the curves in Fig. \ref{fig:k_open}.
\begin{figure}[h!]
	\begin{subfigure}{.3\linewidth}
		\centering
		\includegraphics[width=\textwidth]{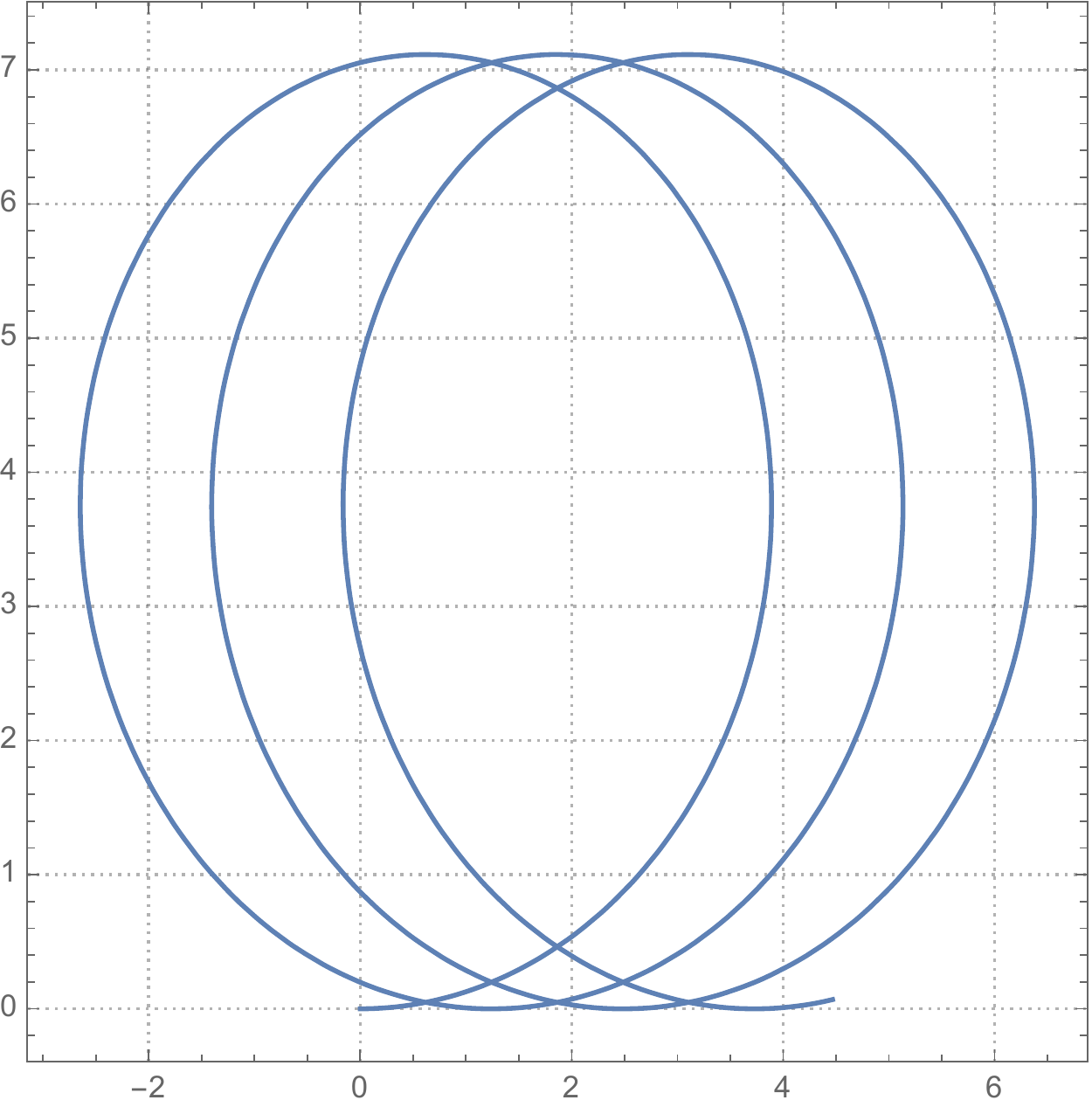}
		\caption{}
		\label{fig:fig6}
	\end{subfigure}
	\begin{subfigure}{.3\linewidth}
		\centering
		\includegraphics[width=\textwidth]{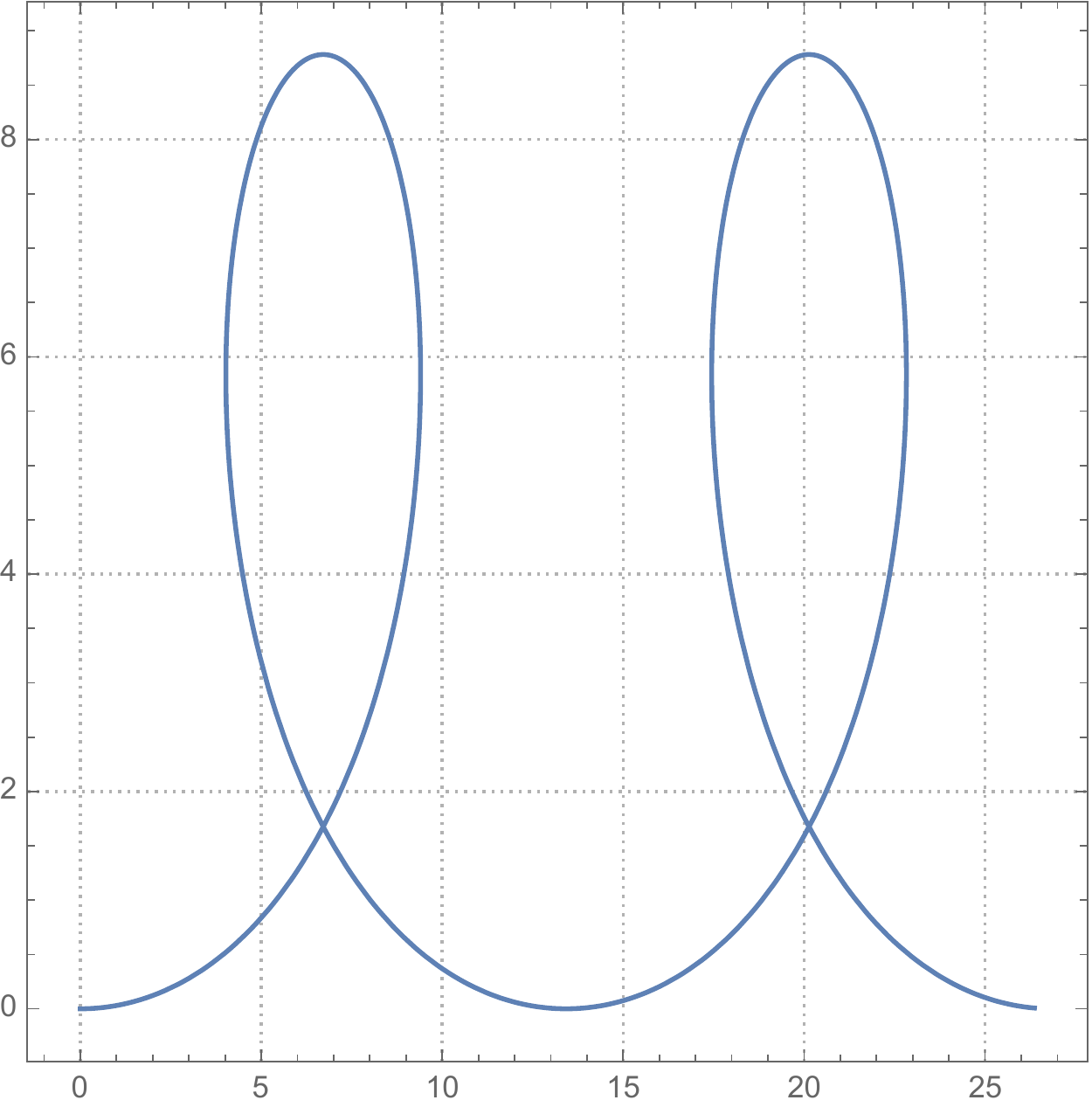}
		\caption{}
		\label{fig:fig5}
	\end{subfigure}
	\begin{subfigure}{.3\linewidth}
	\centering
	\includegraphics[width=\textwidth]{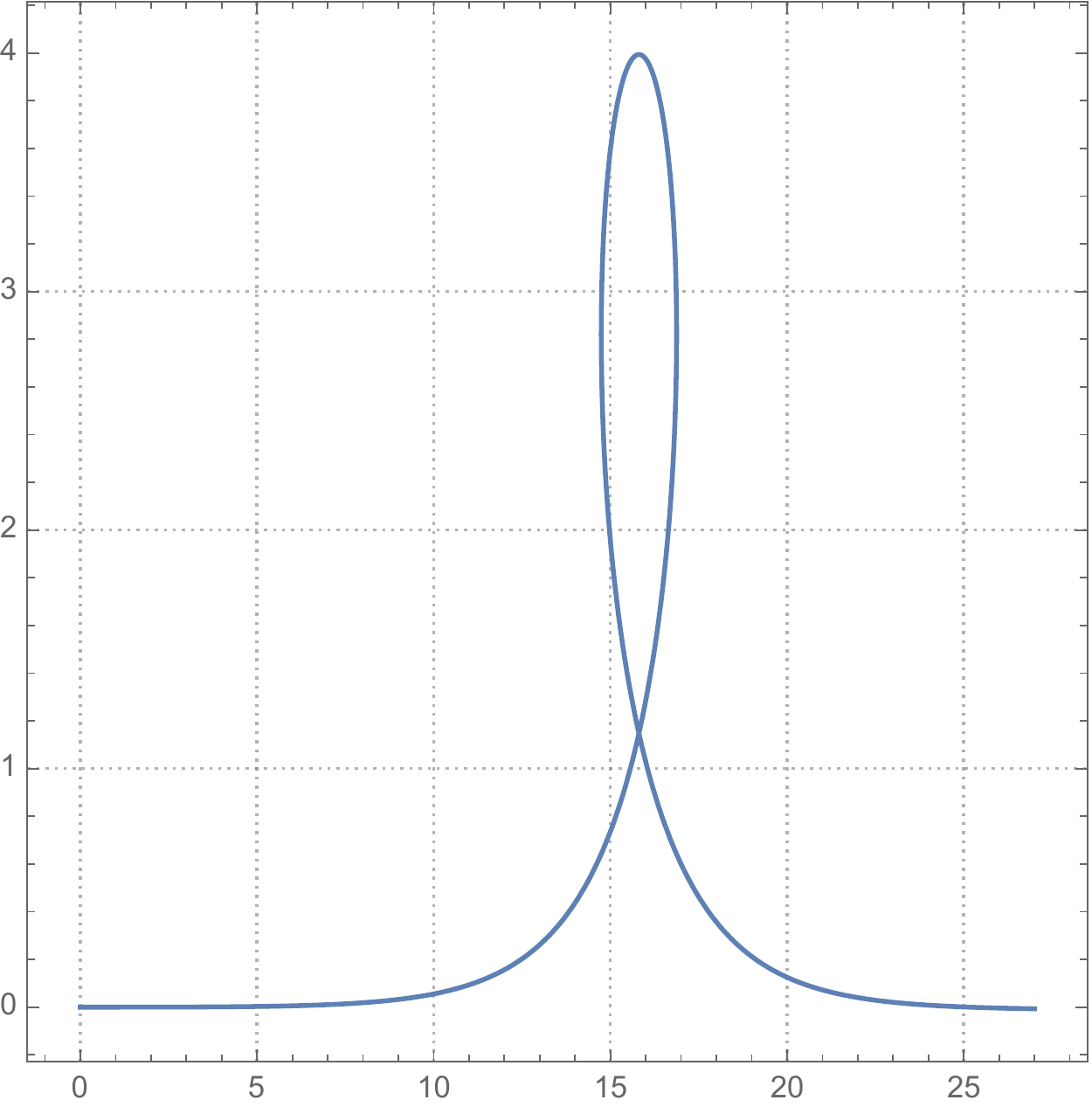}
	\caption{}
	\label{fig:fig1}
\end{subfigure}\\
	\begin{subfigure}{.3\linewidth}
		\centering
		\includegraphics[width=\textwidth]{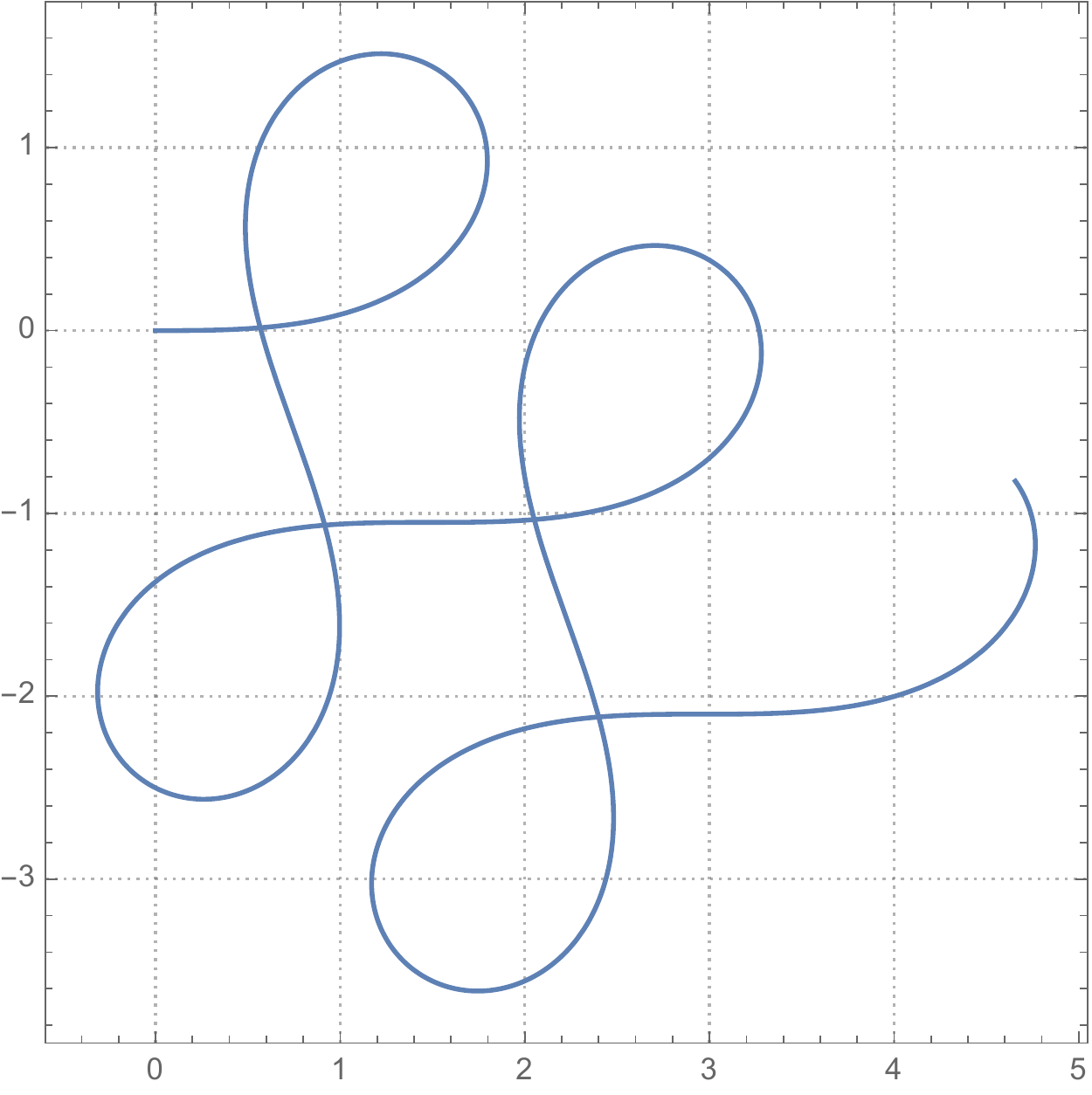}
		\caption{}
		\label{fig:fig2}
	\end{subfigure}
	\begin{subfigure}{.3\linewidth}
		\centering
		\includegraphics[width=\textwidth]{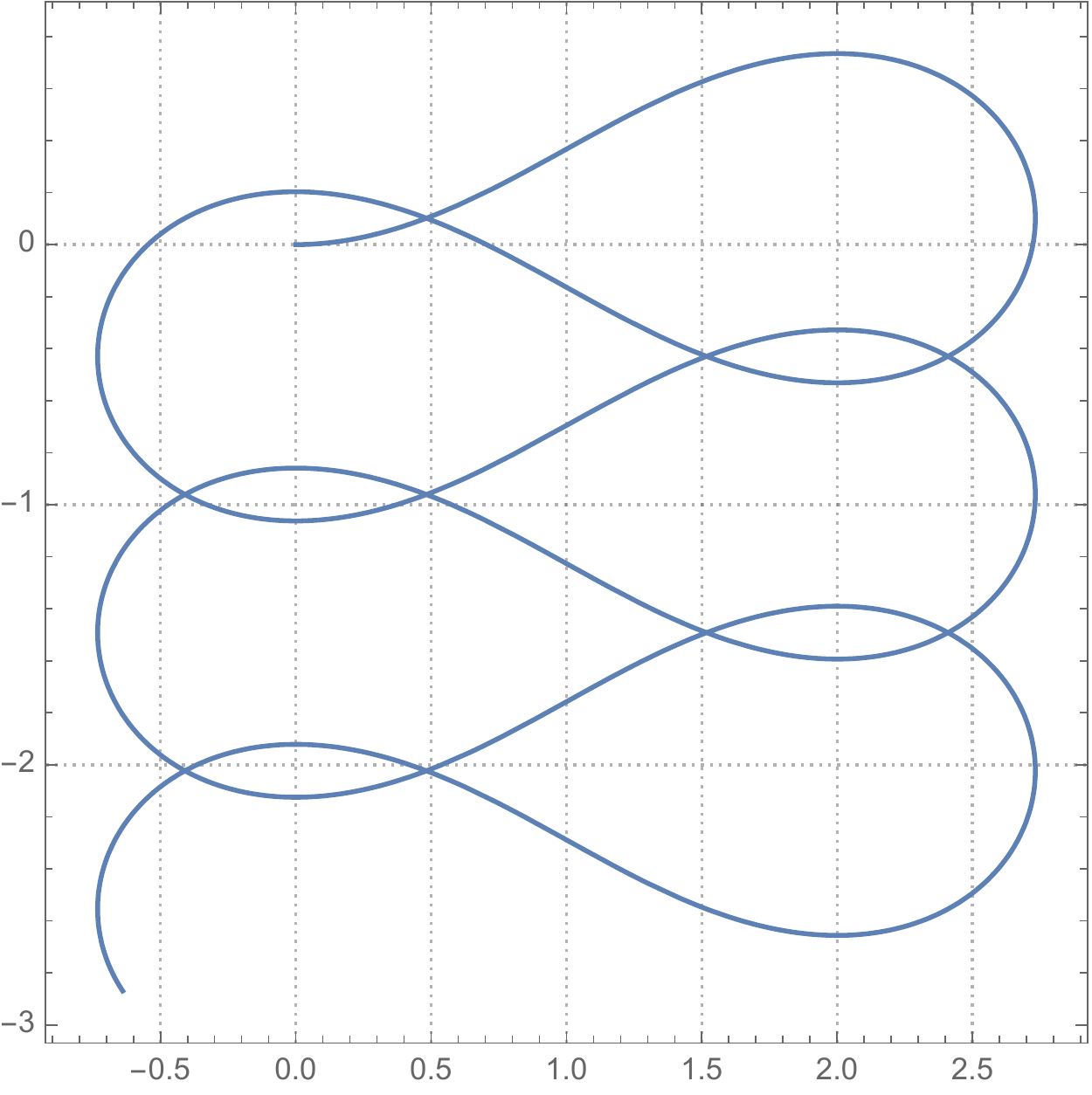}
		\caption{}
		\label{fig:fig4}
	\end{subfigure}
	\begin{subfigure}{.3\linewidth}
	\centering
	\includegraphics[width=\textwidth]{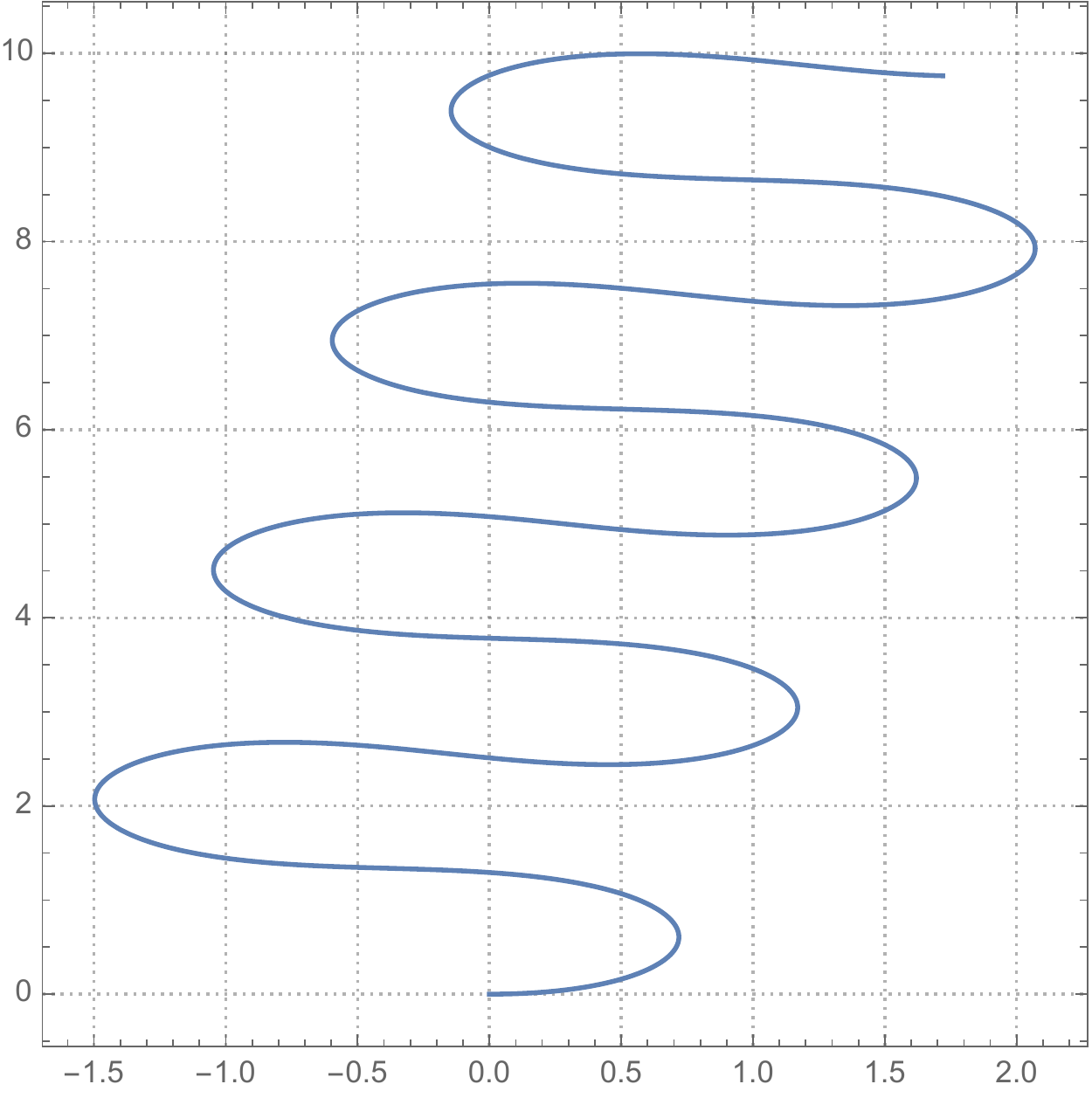}
	\caption{}
	\label{fig:fig8}
\end{subfigure}
	\caption{Free elasticae: all the possibile (planar) shapes \cite{LS_85}.}
	\label{fig:k_open}
\end{figure}

\begin{table}[t!]
	\centering 
	\begin{tabular}{|c|c|c|c|c|}
		\hline
		& $c_1$ &$\k_0$ &$\k_1$ &$L $  \\ \hline
		
		Fig. \ref{fig:fig6}
		&$0.08$
		&$0.25$
		&$0$
		&$22 \pi$
		\\ \hline
		Fig. \ref{fig:fig5}
		&$0.08$
		&$0.06$
		&$0$
		&$21 \pi$
		\\ \hline
		Fig. \ref{fig:fig1}
		&$0.5$
		&$0$
		&$0,001$
		&$8 \pi$
		\\ \hline
		Fig. \ref{fig:fig2}
		&$\sqrt{2}$
		&$0$
		&$0.5$
		&$8 \pi$
		\\ \hline
		Fig. \ref{fig:fig4}
		&$1$
		&$1$
		&$-1$
		&$8 \pi$
		\\ \hline
		Fig. \ref{fig:fig8}
		&$0.3$
		&$0.91$
		&$1.43$
		&$8 \pi$
		\\ \hline
	\end{tabular}
	\caption{Numerical values for free elasticae in Fig.\,\ref{fig:k_open}.}\label{tab:val_k_open}
\end{table}
\noindent
{\it Space curves.} In this case we do not have $c_2=0$, so we have to deal with the complete system \eqref{elastica}. Again, we can look for the curve ${\bm r}\colon [0,L] \to \R^3$ given by 
\[
{\bm r}(s)=\int_0^s {\bm t}\,dr
\]
where now the tangent vector ${\bm t}$ is the solution of 
\[
\left\{\begin{array}{lc}
{\bm t}'=\k{\bm n},\\
{\bm n}'=-\k{\bm t}+\tau{\bm b},\\
{\bm b}'=-\tau{\bm n},\\
{\bm t}(0) = (1,0,0),\\
{\bm n}(0) = (0,1,0),\\
{\bm b}(0) = (0,0,1),
\end{array}\right.
\]
To find closed spatial curves, we use the stop condition 
\begin{equation}\label{eq:stopcond_spazio}
d = |\bm{r}(L)| +  |\bm{t}(L) - (1,0,0)| < 10^{-6},
\end{equation}
similar to the one introduced in the plane. The results for closed elasticae are reported in Fig. \ref{fig:k_tau} and in the corresponding Table \ref{tab:val_k_tau}. Without \eqref{eq:stopcond_spazio} the curves maybe open and we report some examples in Fig.\,\ref{fig:k_tau_open} and Table \ref{tab:val_k_tau_open}.
\begin{figure}[h!]
	\begin{subfigure}{.45\linewidth}
		\centering
		\includegraphics[width=0.85\textwidth]{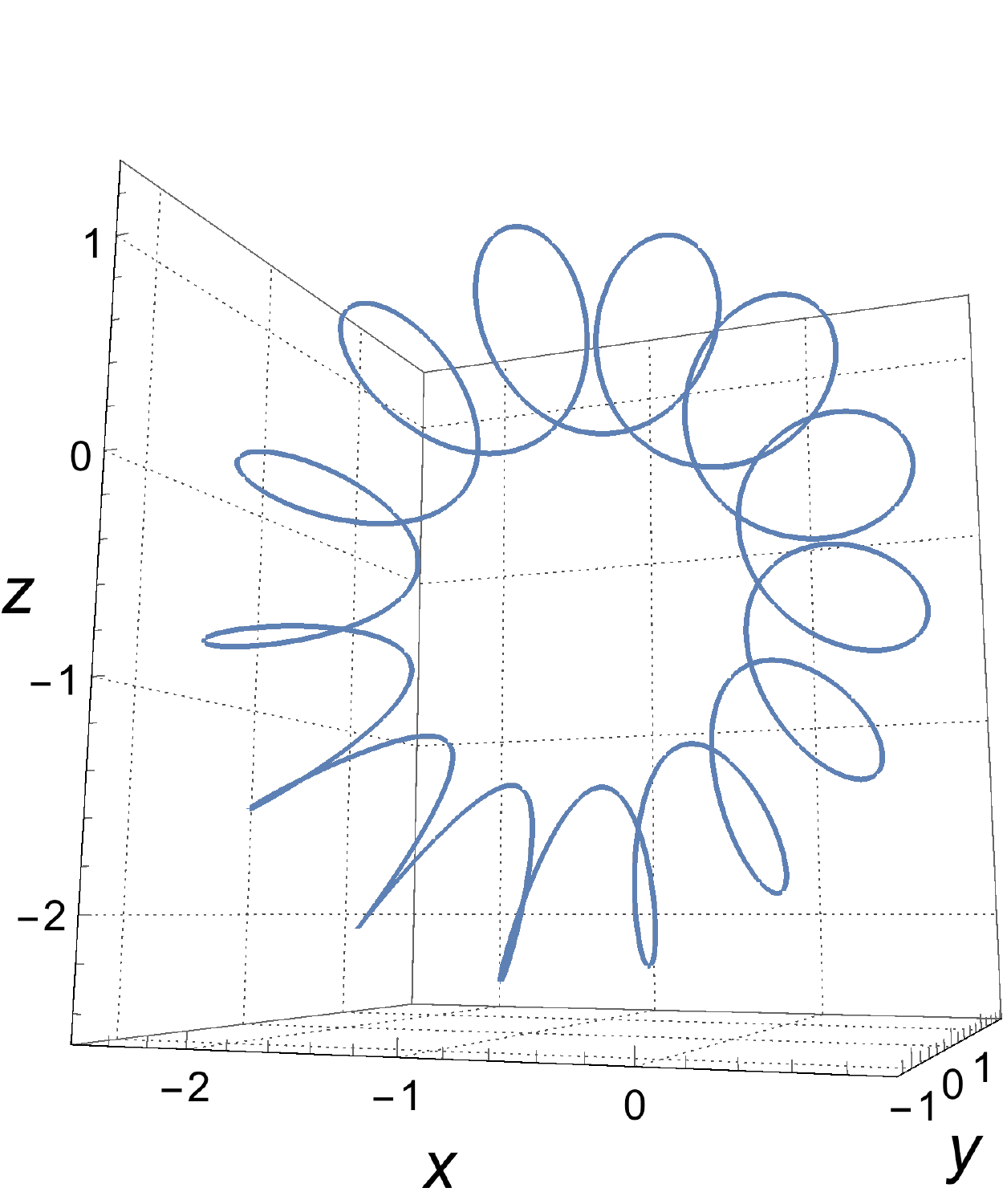}
		\caption{}
		\label{fig:k_tau_fig1}
	\end{subfigure}
	\begin{subfigure}{.48\linewidth}
		\centering
		\includegraphics[width=0.85\textwidth]{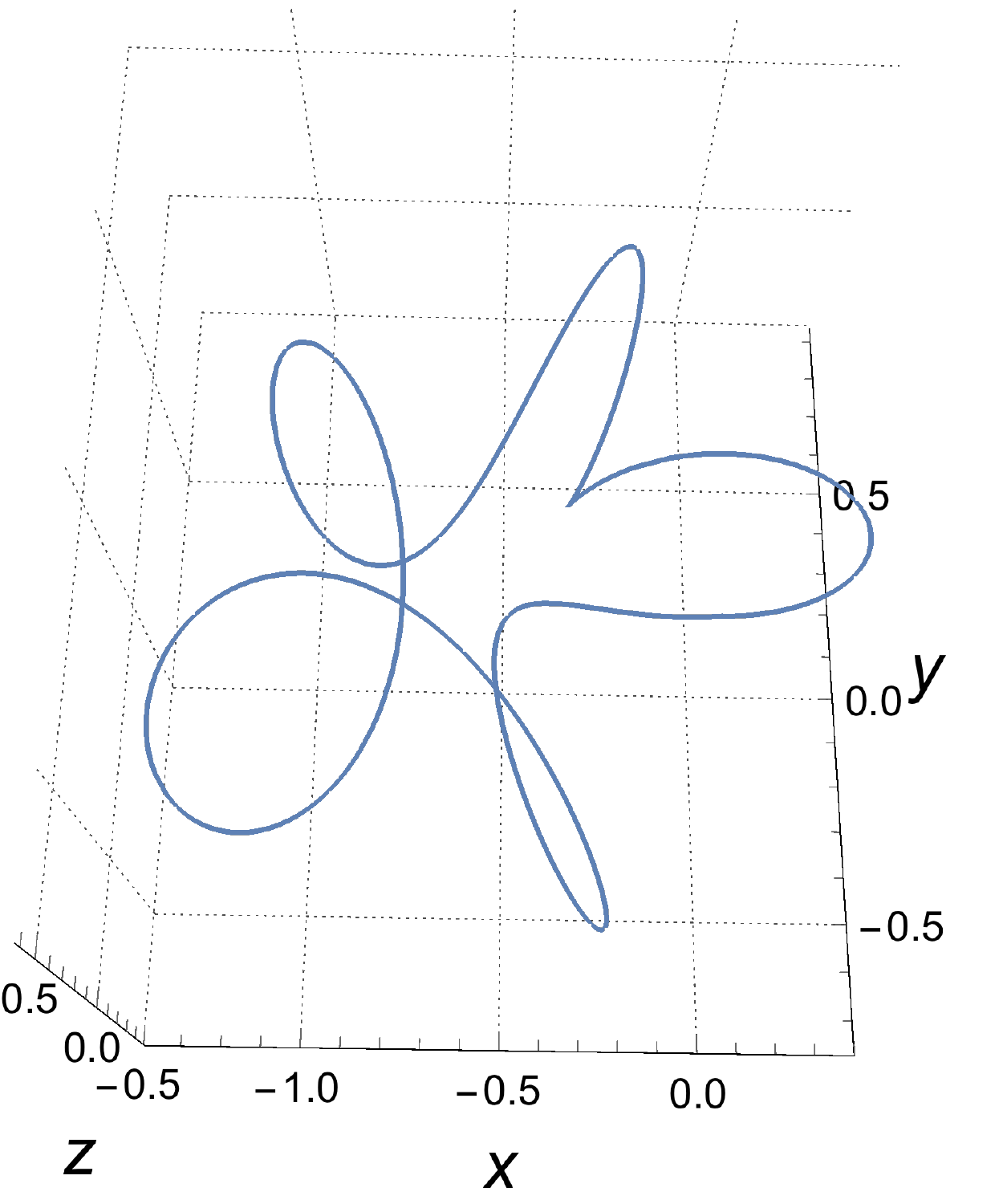}
		\caption{}
		\label{fig:k_tau_fig5}
	\end{subfigure}
\begin{subfigure}{.45\linewidth}
	\centering
	\includegraphics[width=0.85\textwidth]{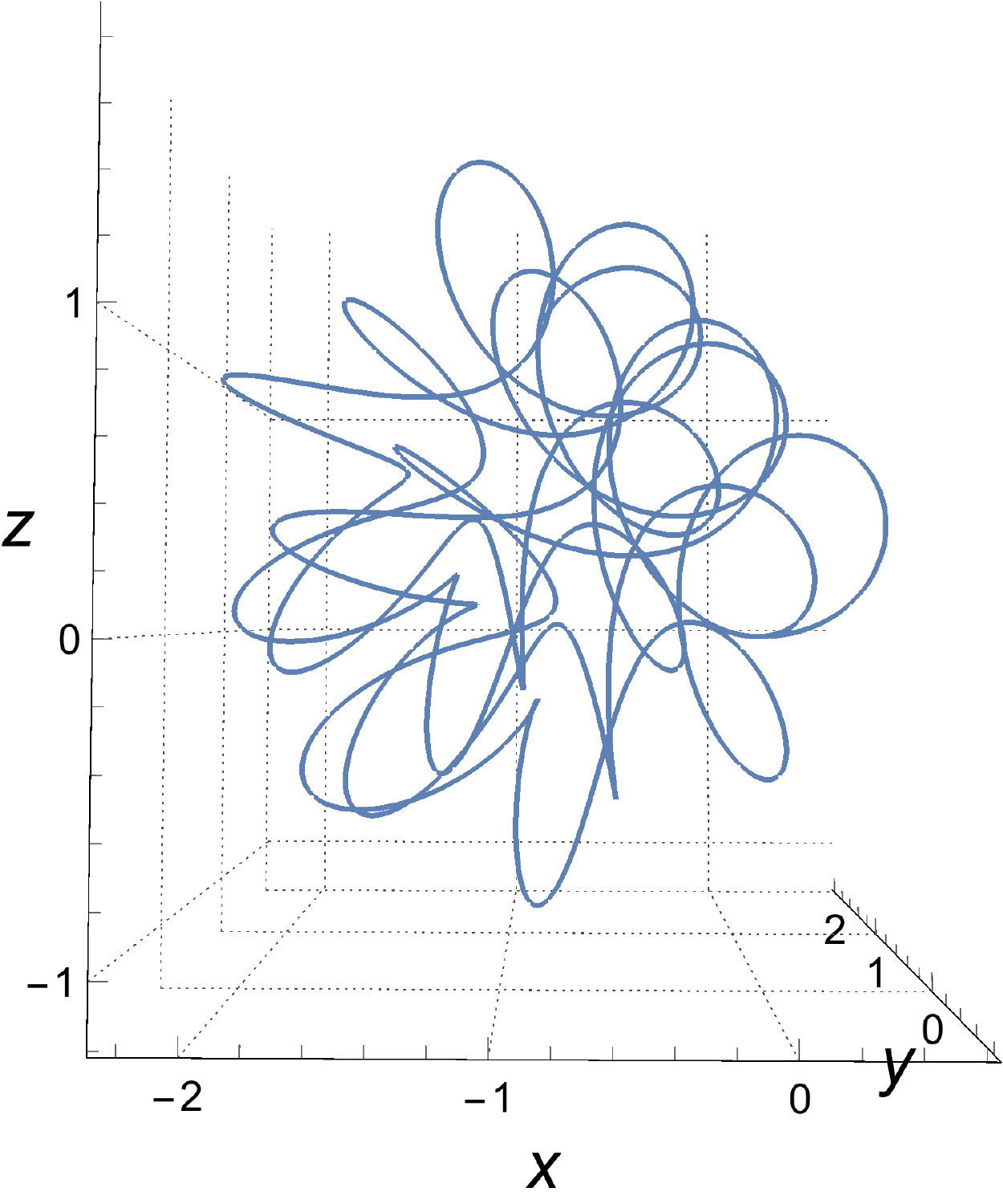}
	\caption{}
	\label{fig:k_tau_fig2}
\end{subfigure}
\begin{subfigure}{.48\linewidth}
	\centering
	\includegraphics[width=0.85\textwidth]{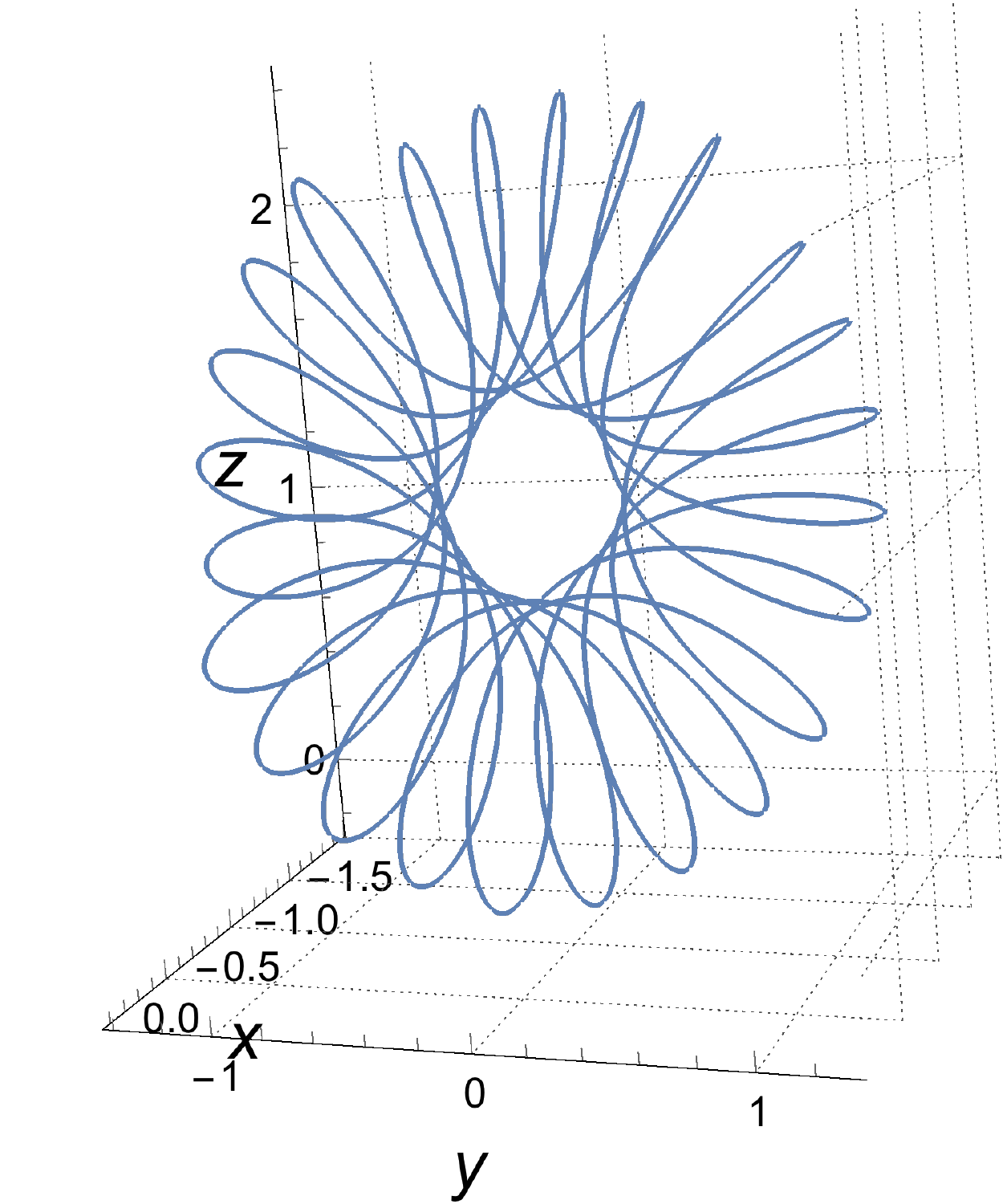}
	\caption{}
	\label{fig:k_tau_fig4}
\end{subfigure}
	\caption{ Examples of closed elasticae in the space.}
	\label{fig:k_tau}
\end{figure}

\begin{table}[t!]
	\centering 
	\begin{tabular}{|c|c|c|c|c|c|}
		\hline
		&$c_1$ & $c_2$ &$\k_0$ &$\k_1$ &$L $  \\ \hline
		
		Fig. \ref{fig:k_tau_fig1}
		&$1.25316$
		&$3.92702$
		&$1.58313$
		&$0.528316$
		&$ 16 \pi$
	
		\\ \hline
		Fig. \ref{fig:k_tau_fig5}
		&$0.08$
		&$5.06$
		&$2.53458$
		&$4.04$
		&$3 \pi$
	
		\\ \hline
		Fig. \ref{fig:k_tau_fig2}
		&$2.06465$
		&$4.38778$
		&$1.51781$
		&$1.47094$
		&$16 \pi$

		\\ \hline
		Fig. \ref{fig:k_tau_fig4}
		&$1.62767$
		&$4.08942$
		&$2.85503$
		&$0.669953$
		&$30 \pi$

		\\ \hline
	\end{tabular}
	\caption{Numerical values chosen to plot the elasticae in Fig.\,\ref{fig:k_tau}.}\label{tab:val_k_tau}
\end{table}

\begin{figure}[h!]
	\begin{subfigure}{.48\linewidth}
		\centering
		\includegraphics[width=0.75\textwidth]{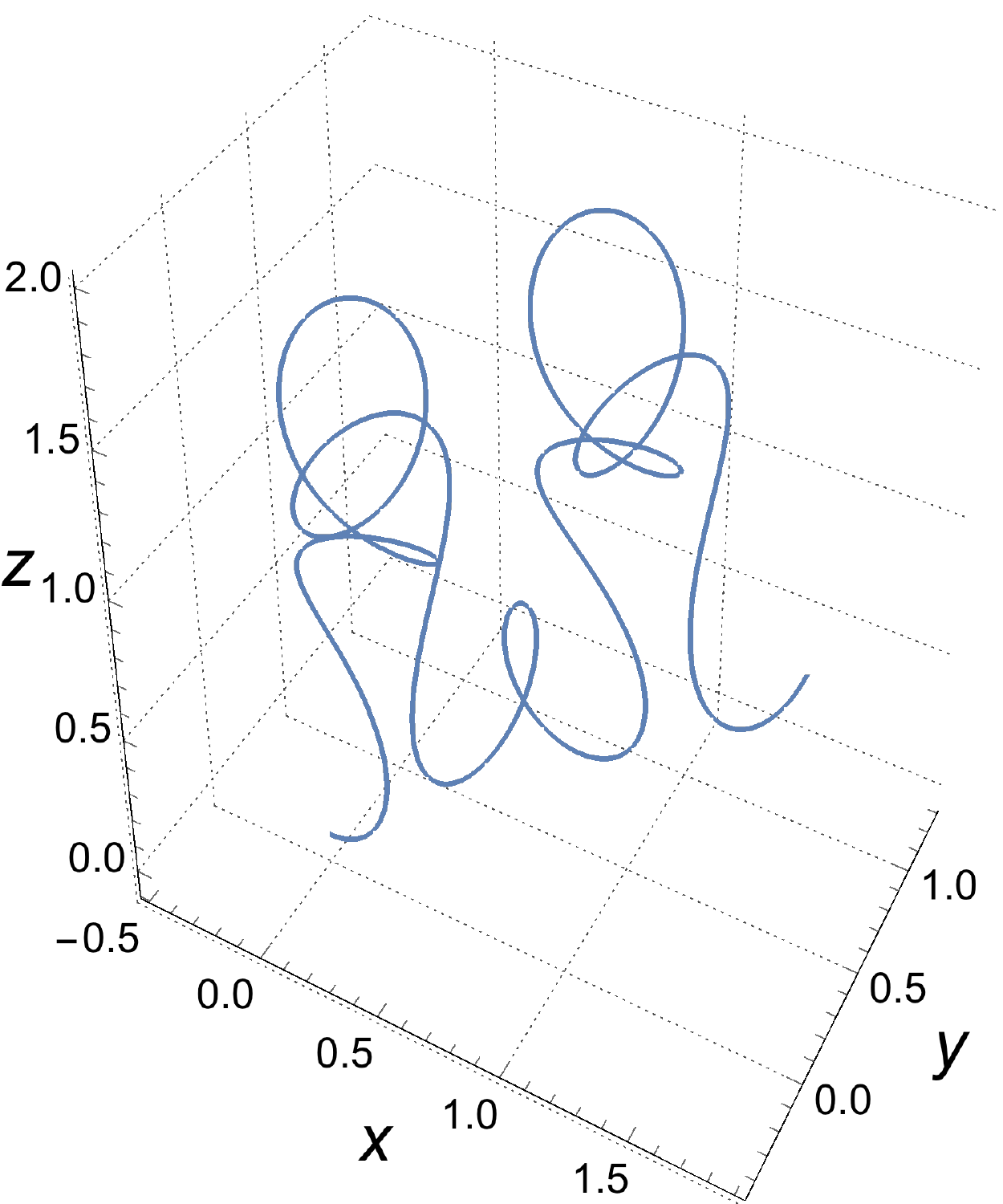}
		\caption{}
		\label{fig:k_tau_fig1_open}
	\end{subfigure}
	\begin{subfigure}{.48\linewidth}
		\centering
		\includegraphics[width=0.75\textwidth]{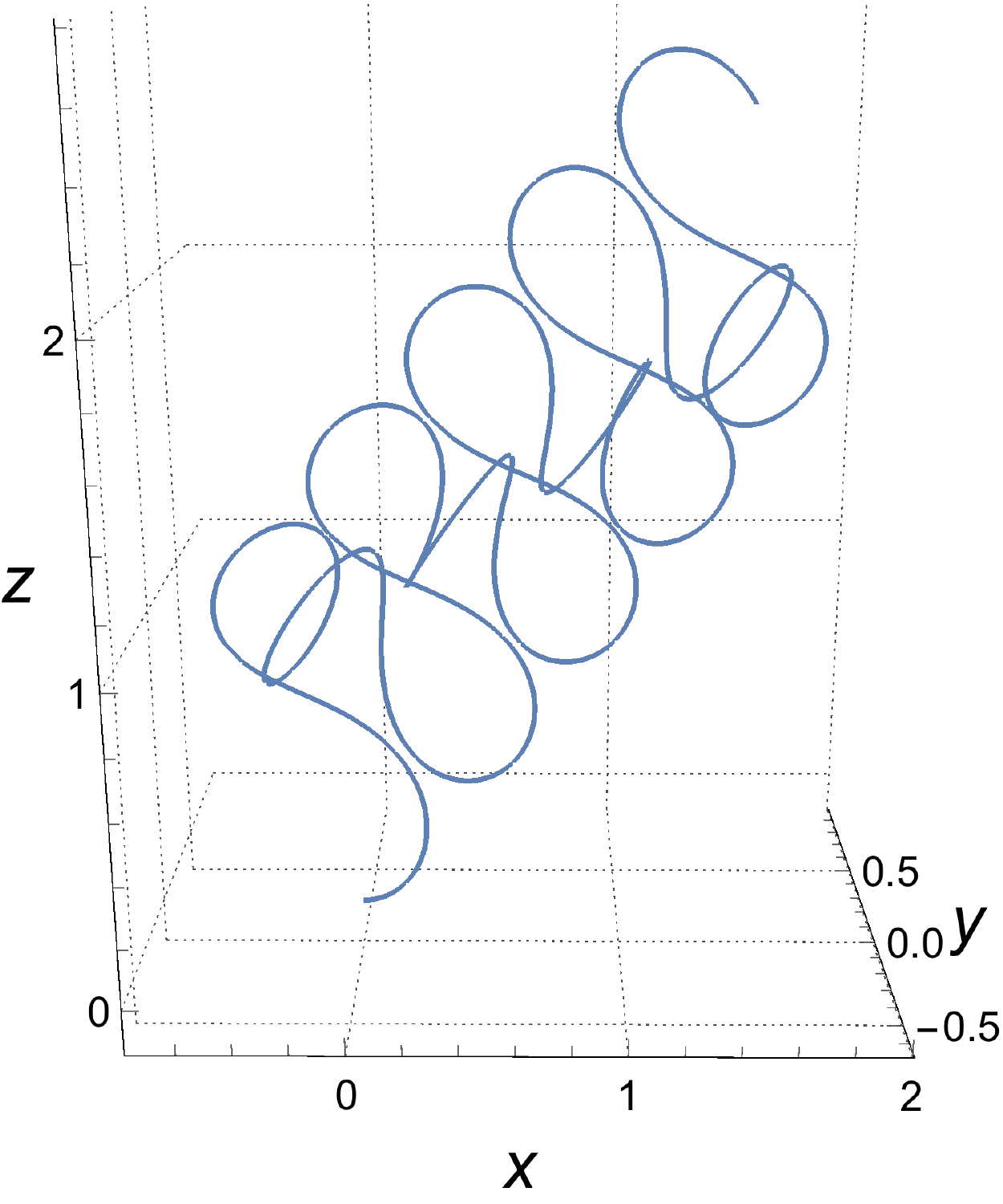}
		\caption{}
		\label{fig:k_tau_fig2_open}
	\end{subfigure}\\
	\begin{subfigure}{.48\linewidth}
		\centering
		\includegraphics[width=0.75\textwidth]{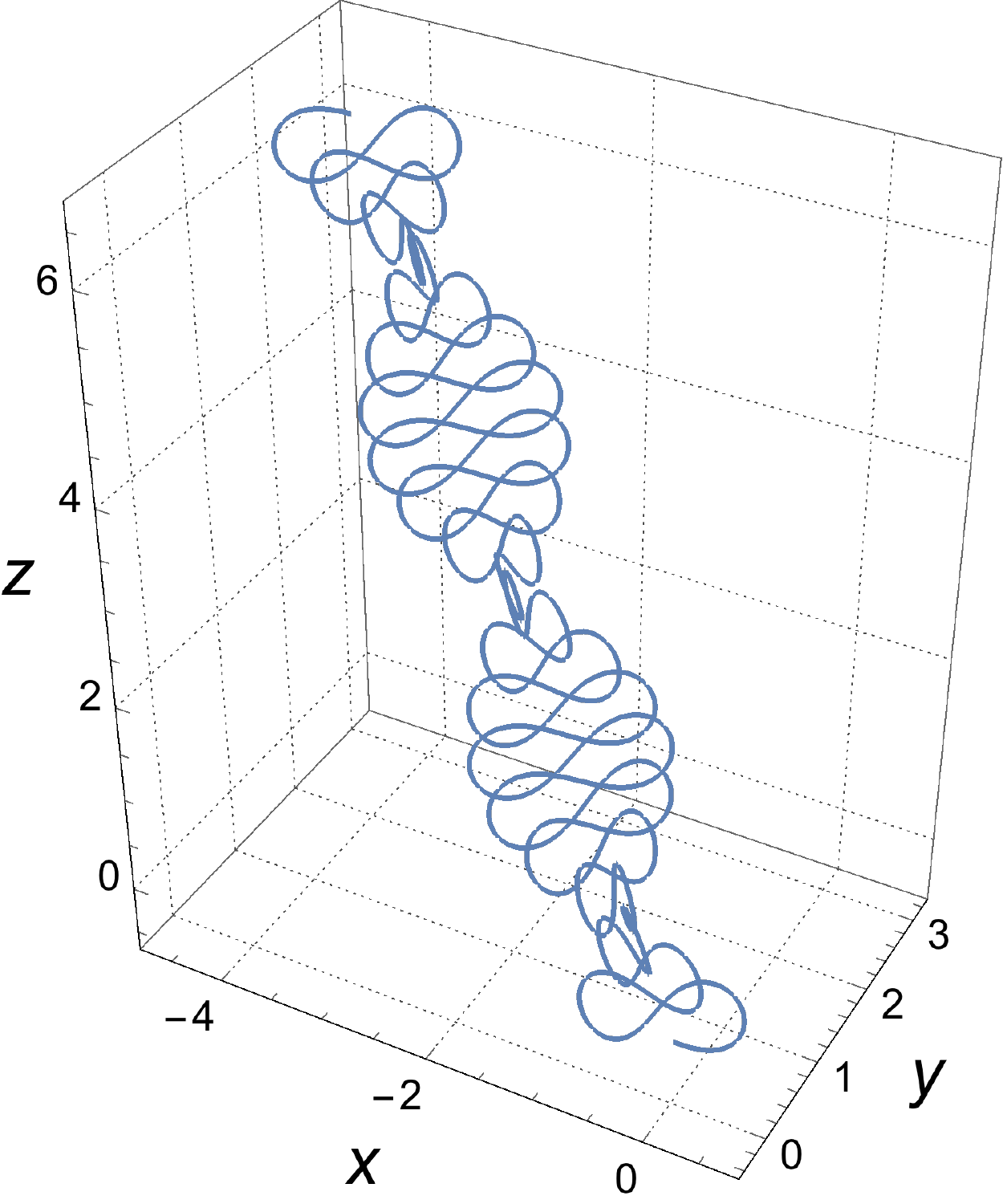}
		\caption{}
		\label{fig:k_tau_fig3_open}
	\end{subfigure}
	\begin{subfigure}{.48\linewidth}
		\centering
		\includegraphics[width=0.75\textwidth]{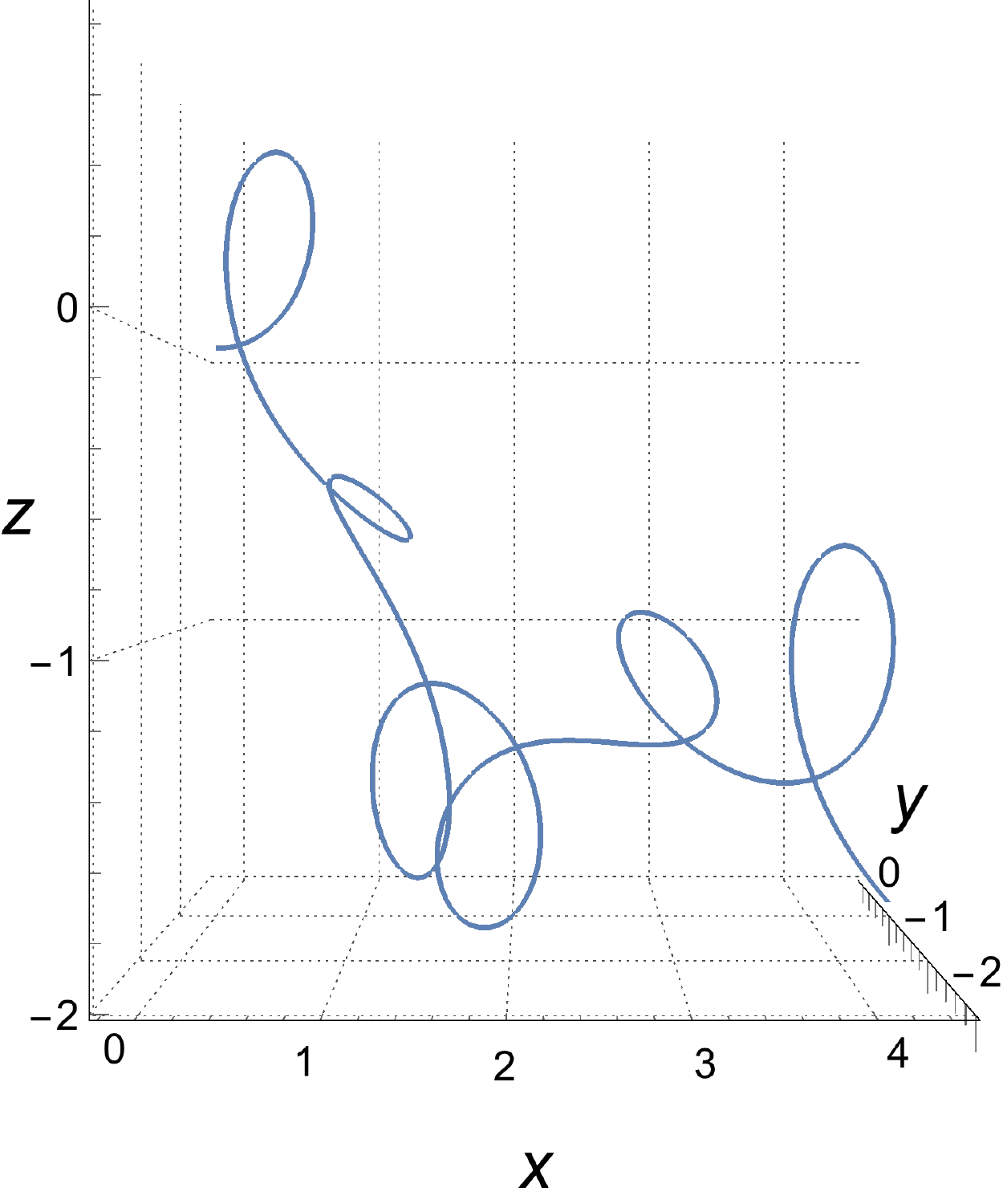}
		\caption{}
		\label{fig:k_tau_fig4_open}
	\end{subfigure}
	\caption{ Examples of open regular elasticae in the space.}
	\label{fig:k_tau_open}
\end{figure}

\begin{table}[t!]
	\centering 
	\begin{tabular}{|c|c|c|c|c|c|}
		\hline
		&$c_1$ & $c_2$ &$\k_0$ &$\k_1$ &$L $ \\ \hline
		
		Fig. \ref{fig:k_tau_fig1_open}
		&$6.85389$
		&$7.80699$
		&$3.97779$
		&$1.48377$
		&$6 \pi$
		\\ \hline
		Fig. \ref{fig:k_tau_fig2_open}
		&$3.76699$
		&$6.25666$
		&$3.78912$
		&$3.05338$
		&$8 \pi$
		\\ \hline
		Fig. \ref{fig:k_tau_fig3_open}
		&$0.700994$
		&$4.99512$
		&$0.478556$
		&$3.74012$
		&$25 \pi$
		\\ \hline
		Fig. \ref{fig:k_tau_fig4_open}
		&$0.794876$
		&$6.57172$
		&$0.899095$
		&$1.12241$
		&$7 \pi$
		\\ \hline
	\end{tabular}
	\caption{Numerical values chosen to plot the elasticae in Fig. \ref{fig:k_tau_open}.}\label{tab:val_k_tau_open}
\end{table}

\subsection{The model case in the biological applications}
\color{black}
One of the most common energies arising from the applications to Biophyiscs \cite{SH} takes the form 
\begin{equation}
\label{eq:fun_k2tau2}
\int_0^L\frac{\k^2+\tau^2}{2}\,ds.
\end{equation}
Following our notation, let 
\[
f(s,a,b)=f(a,b)=\frac{a^2+b^2}{2}.
\]
It is straightforward to see that $f$ satisfies all the assumptions \eqref{f0}-\eqref{f1}-\eqref{f2}-\eqref{f3}-\eqref{f4} with the choice $p=2$. As a consequence we have existence of minimizers. Of course, since plane curves have $\tau=0$ as in the case of the Euler elastica the circumference of length $L$ is still a minimizer. Concerning critical points, first of all, we notice that $f$ is smooth and we have $f_a=a$ and $f_b=b$. Then the system \eqref{elw} reads as 
\begin{equation}\label{eq:system_k_tau}
\left\{\begin{array}{lll}
\tau'=\mu\k,\\
{\bm\lambda} \cdot {\bm n}=-\k'-\mu\tau,\\
{\bm\lambda} \cdot {\bm b}=\mu'.
\end{array}\right.
\end{equation}
It is easy to see that the solutions $\k,\tau$ of that system are smooth and we can therefore apply Theorem \ref{main3}.  Then, on any interval $I$ where $\k\ne 0$ we have 
\begin{equation}\label{membranes}
\left\{\begin{array}{ll}
\displaystyle \tau \k'-\left(\frac{\tau'}{\k}\right)''+\frac{\tau^2\tau'}{\k}=0\\
\\
\displaystyle -\k\k'-\tau\tau'=\left(\frac{\k''}{\k}+\frac{2\tau}{\k}\left(\frac{\tau'}{\k}\right)'+\frac{(\tau')^2}{\k^2}\right)'.
\end{array}\right.
\end{equation}
These equations look quite difficult and hold true only when $\k\ne 0$. The only fact we put in evidence is that \eqref{membranes}$_2$ can be rewritten as 
\[
-\left(\frac{\k^2+\tau^2}{2}\right)'=\left(\frac{\k''}{\k}+\frac{2\tau}{\k}\left(\frac{\tau'}{\k}\right)'+\frac{(\tau')^2}{\k^2}\right)'\, ,
\]
that is 
\[
\frac{\k''}{\k}+\frac{2\tau}{\k}\left(\frac{\tau'}{\k}\right)'+\frac{(\tau')^2}{\k^2}+\frac{\k^2+\tau^2}{2}=c
\]
for some constant $c$. At the end we can reduce the analysis to 
\begin{equation}\label{membranes2}
\left\{\begin{array}{ll}
\displaystyle \tau \k'-\left(\frac{\tau'}{\k}\right)''+\frac{\tau^2\tau'}{\k}=0,\\
\\
\displaystyle \frac{\k''}{\k}+\frac{2\tau}{\k}\left(\frac{\tau'}{\k}\right)'+\frac{(\tau')^2}{\k^2}+\frac{\k^2+\tau^2}{2}=c.
\end{array}\right.
\end{equation}
Concerning the numerical analysis we tried to reproduce the same arguments to solve \eqref{membranes2}. However, we did not succeed in varying randomly the constant $c$ and the initial conditions for $\k,\tau$, since it is very likely that $\k = 0$ at some time, thus stopping the numerical procedure. We show the results in Fig.\,\ref{fig:k_tau_int_nozero} and in Table \ref{tab:k_tau_open}.

\begin{figure}[h!]
	\begin{subfigure}{.31\linewidth}
		\centering
		\includegraphics[width=\textwidth]{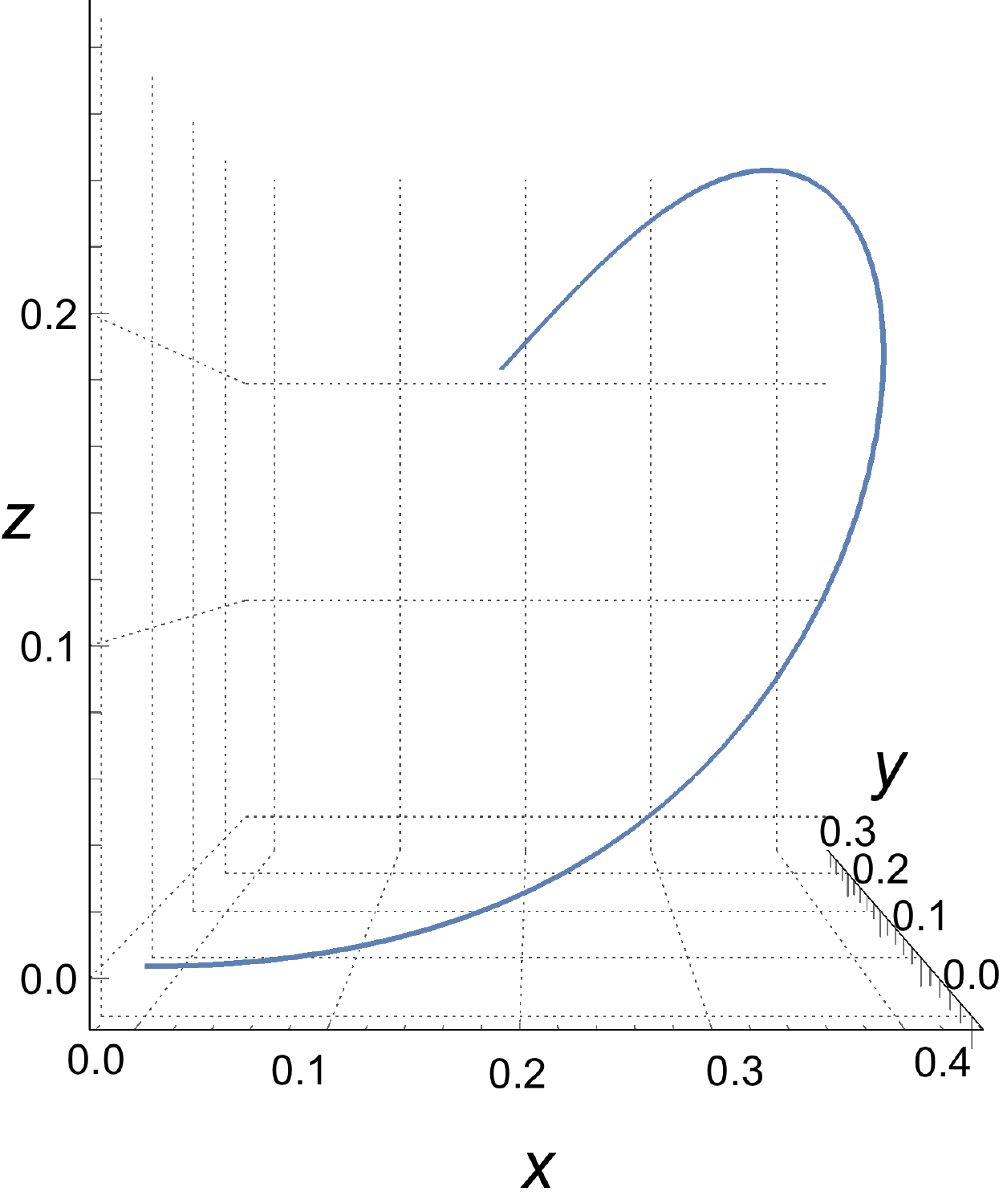}
		\caption{}
		\label{fig:k_tau_1}
	\end{subfigure}
	\begin{subfigure}{.31\linewidth}
		\centering
		\includegraphics[width=\textwidth]{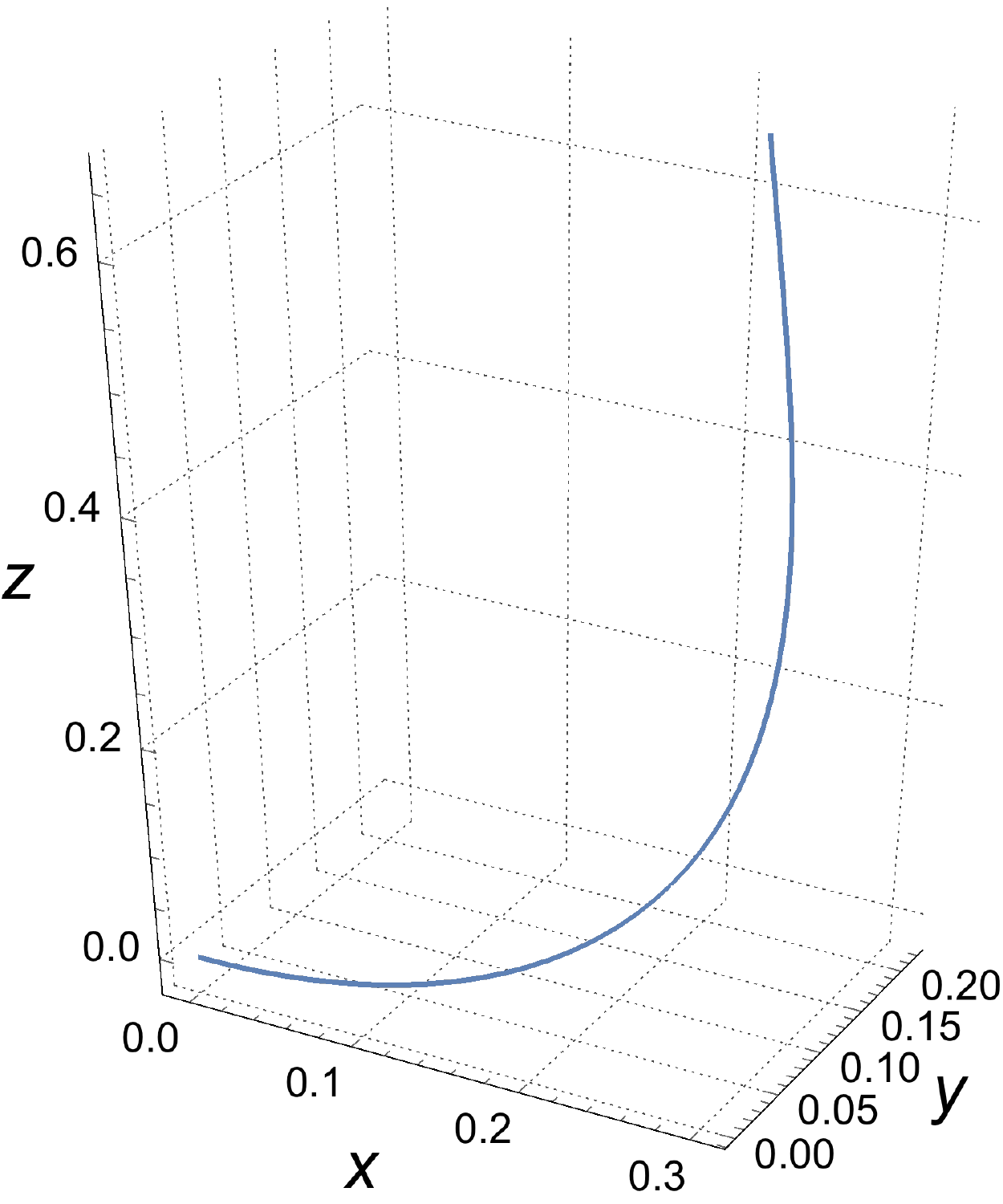}
		\caption{}
		\label{fig:k_tau_2}
	\end{subfigure}
	\begin{subfigure}{.31\linewidth}
		\centering
		\includegraphics[width=\textwidth]{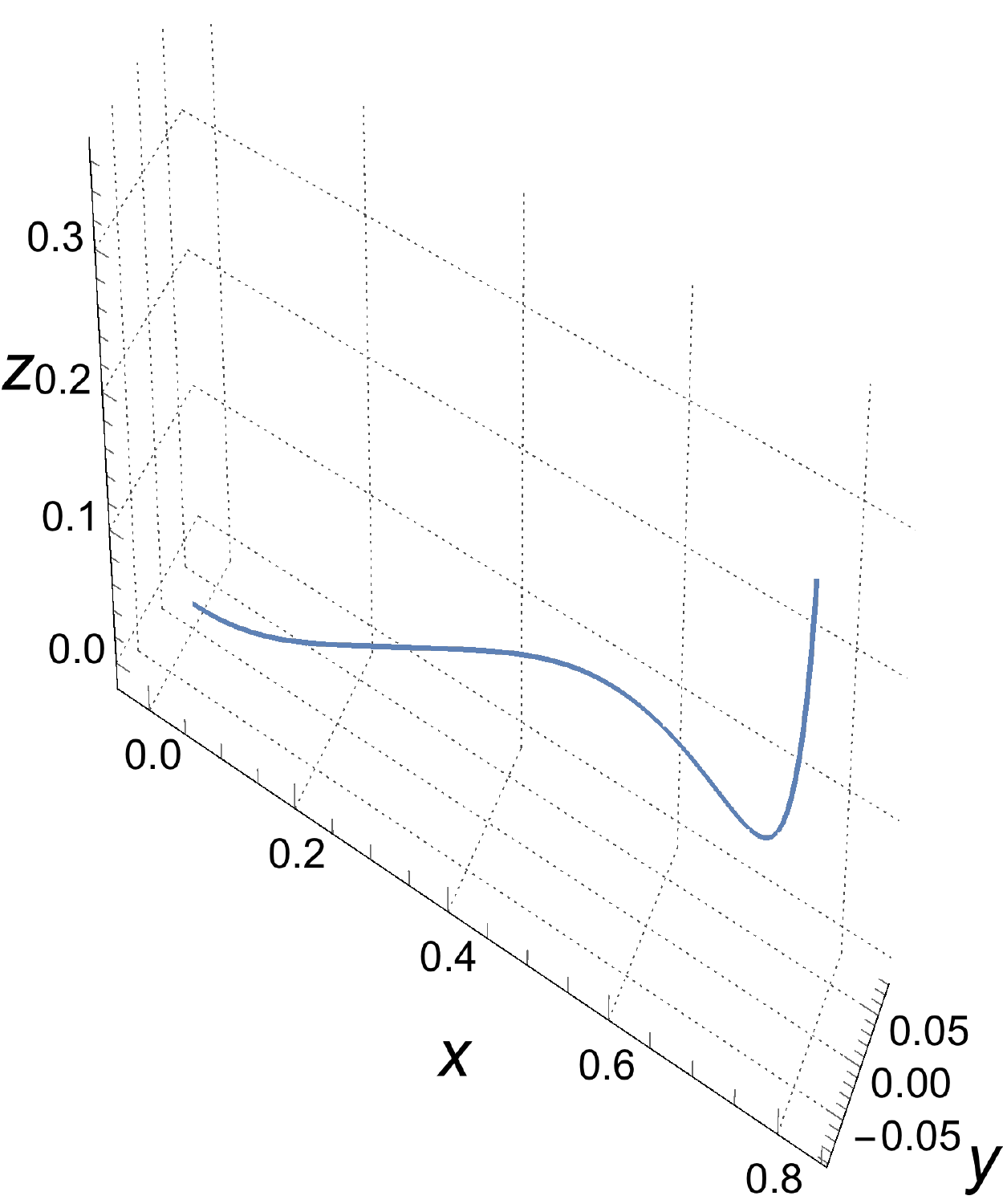}
		\caption{}
		\label{fig:k_tau_3}
	\end{subfigure}
	\caption{Some critical curves for the energy functional \eqref{eq:fun_k2tau2}.}
	\label{fig:k_tau_int_nozero}
\end{figure}

\begin{table}[t!]
	\centering 
	\begin{tabular}{|c|c|c|c|c|c|c|c|}
		\hline
		&$c$ & $\k_0$ &$\k_1$ &$\tau_0$ &$\tau_1$ &$\tau_2$ &$L$  \\ \hline
		
		Fig. \ref{fig:k_tau_1}
		&$-0.1$
		&$0.787616$
		&$3.33006$
		&$1.00144$
		&$ 4.69347$
		& $4.29121$
		& $0.929236143$
		\\ \hline
		Fig. \ref{fig:k_tau_2}
		&$0.01$
		&$3.05775$
		&$4.22982$
		&$0.749952$
		&$0.997559$
		& $3.02353$
		& $0.8734864103$
		\\ \hline
		Fig. \ref{fig:k_tau_3}
		&$1.03$
		&$1.95093$
		&$1.6048$
		&$8.21105$
		&$0.508862$
		& $3.25462$
		& $1.0579039889$
		\\ \hline
	\end{tabular}
	\caption{Numerical values chosen to plot the critical curves in  Fig.\,\ref{fig:k_tau_int_nozero}.}\label{tab:k_tau_open}
\end{table}

\subsection{The Sadowsky energy functional} \label{par:sadowsy}

 Following our notation, the Sadowsky energy functional takes the form 
\[
\int_0^L \frac{(\k^2+\tau^2)^2}{\k^2}\,ds.
\]
Some authors tried to give a rigorous justification of the limit process. We mention only the paper by Freddi et al.\,\cite{FHMP} (further references can be found therein), where the authors studied the $\Gamma$-limit of the bending energy on the M\"obius strip with respect to a topology that ensures the convergence of the minimizers. In this way, they obtain as $\Gamma$-limit the functional 
\[
\int_0^L f(\k,\tau)\,ds
\]
where $f\colon \R^2 \to \R$ is given by 
\begin{equation}\label{eq:sadowsky}
f(a,b)=\left\{\begin{array}{ll}
\displaystyle \frac{(a^2+b^2)^2}{a^2} & \text{if $|a|>|b|$},\\
\\
4b^2  & \text{if $|a|\le |b|$}.
\end{array}\right.
\end{equation}
Such a functional turns out to be the {\it corrected version} of the Sadowsky functional. It is easy to see that $f$ is continuous and convex: in order to see the convexity, we notice that that if $f_1\colon \mathbb R^2\setminus\{(a,b) \in \R^2 : a>0\} \to \R$ and $f_2 \colon \R^2 \to \R$ are given by
\[
f_1(a,b)= \frac{(a^2+b^2)^2}{a^2}, \quad f_2(a,b)=4b^2,
\]
then $f_1,f_2$ are both of class $C^1$ and convex and $\nabla f_1(a,b)=\nabla f_2(a,b)$ whenever $a=|b|>0$. Then, \eqref{f1} is satisfied. Moreover,  
\[
\frac{(a^2+b^2)^2}{a^2}=a^2+2b^2+\frac{b^4}{a^2}\ge a^2+2b^2 
\]
while if $|a|\le |b|$ we get $4b^2 \ge 2b^2+2a^2\ge a^2 +2b^2$. As a consequence, 
\[
f(a,b)\ge a^2+2b^2
\]
which shows \eqref{f2} with $p=2$. We are therefore in position to apply Theorem \ref{main1}, hence we have existence of minimizers. 

\begin{remark}
{\rm The system of critical points for the Sadowsky functional takes a very complicated form due to the fact that the energy density Eq. \eqref{eq:sadowsky} depends on the value of the curvature. Hence, we are not going to write explicitly the first-order conditions for minimizers since we think that the pure numerically investigation of this kind of system is out of the purpose of this paper.}
\end{remark}

\section*{Acknowledgments}
The authors wish to thank Antonio De Rosa, Maria Giovanna Mora and Alessandra Pluda for helpful suggestions and fruitful discussions. The work of GB was partially supported by GNFM-INdAM though the GNFM ``Progetto Giovani" 2020 {\em Transizioni di forma nella materia biologica e attiva}. The work of LL has been partially supported by INdAM - GNAMPA. AM acknowledges the support from MIUR, PRIN 2017 Research Project {\em ``Mathematics of active materials: from mechanobiology to smart devices".} The work of AM has been also partially supported by INdAM - GNFM.

\end{document}